\documentclass[11pt]{article}
\usepackage{CJK}
\usepackage{amsmath,amsfonts,mathrsfs,amssymb}
\usepackage{indentfirst}
\usepackage{amsthm}
\numberwithin{equation}{section}
\usepackage{color}

\usepackage{hyperref,amssymb,cite}
\usepackage{cite}
\usepackage[misc,geometry]{ifsym}  
\hypersetup{hypertex=true,
	colorlinks=true,
	linkcolor=blue,
	anchorcolor=blue,
	citecolor=blue}

\theoremstyle{plain}
\newtheorem{theorem}{Theorem}[section]
\newtheorem{corollary}[theorem]{Corollary}

\newtheorem{lemma}[theorem]{Lemma}

\theoremstyle{definition}

\numberwithin{equation}{section}

\newcommand{\HT}{\CJKfamily{hei}}

\newcommand{\R}{\mathbb{R}}
\newcommand{\N}{\mathbb{N}}

\def\be{\begin{equation}}
\def\ee{\end{equation}}
\def\bea{\begin{eqnarray}}
\def\eea{\end{eqnarray}}
\def\bes{\begin{eqnarray*}}
\def\ees{\end{eqnarray*}}

\def\y{\begin{eqnarray*}}
\def\ey{\end{eqnarray*}}

\begin{document}
\title{\HT {Sharp anisotropic singular Trudinger-Moser inequalities in the entire space}}
\author{\small {Kaiwen Guo$^a$, Yanjun Liu$^b$ }\\
\small $^a$Chern Institute of Mathematics, Nankai University, Tianjin  300071,  P. R. China\\
\small $^b$School of Mathematical Sciences, Chongqing Normal University,\\
\small Chongqing 401331, P. R. China}
\date{}
\maketitle
\footnote[0]{~~~~~~Kaiwen Guo}
\footnote[0]{~~~~~~gkw17853142261@163.com }
\footnote[0]{\Letter ~ Yanjun Liu}
\footnote[0]{~~~~~~liuyj@mail.nankai.edu.cn }

\noindent{\small
{\bf Abstract:} In this paper, we investigate sharp singular Trudinger-Moser inequalities involving the anisotropic Dirichlet norm $\left(\int_{\Omega}F^{N}(\nabla u)\;\mathrm{d}x\right)^{\frac{1}{N}}$ in the Sobolev-type space $D^{N,q}(\R^{N})$, $q\geq 1$, here $F:\R^{N}\rightarrow[0,+\infty)$ is a convex function of class $C^{2}(\R^{N}\setminus\{0\})$, which is even and positively homogeneous of degree 1, its polar $F^{0}$ represents a Finsler metric on $\R^{N}$. Combing with the connection between convex symmetrization and Schwarz symmetrization, we will establish anisotropic singular Trudinger-Moser inequalities and discuss their sharpness under several different situations, including the case $\|F(\nabla u)\|_{N}\leq 1$, the case $\|F(\nabla u)\|_{N}^{a}+\|u\|_{q}^{b}\leq 1$, and whether they are associated with exact growth.
}

\section{Introduction and main results}\label{section 1}
 In the limiting case, the Sobolev embeddings are replaced by the Trudinger-Moser inequalities. Let $\Omega$ be a bounded domain in $\R^{N}$, $N\geq 2$ and $kp<N$, it is well known that $W_{0}^{k,p}(\Omega)\subset L^{q}(\Omega)$ for all $q$ with $1\leq q\leq{Np}/{(N-kp)}$. However, $W_{0}^{k,\frac{N}{k}}(\Omega)\not\subset L^{\infty}(\Omega)$. For instance, one could verify that the unbounded function $\log\log(1+\frac{1}{\lvert x\rvert})$ is in $W^{1,N}(B_{1}(0))$ for any $N\geq 2$. Trudinger \cite{30}(see also Yudovich \cite{10}) demonstrated, however, that $W_{0}^{1,N}(\Omega)\subset L_{\varphi_{N}}(\Omega)$ in this situation, where $L_{\varphi_{N}}(\Omega)$ is the Orlicz space associated with the Young function $\varphi_{N}(t)=\exp(\alpha\lvert t\rvert^{\frac{N}{N-1}})-1$ for some $\alpha>0$. And these results were improved in \cite{23} by  Moser. Indeed, Moser established that for $0<\alpha\leq\alpha_{N}=N^{\frac{N}{N-1}}\omega_{N}^{\frac{1}{N-1}}$, where $\omega_{N}$ is the volume of the unit $N$-ball, there exists a constant $C=C_{N,\alpha}>0$ satisfying
$$
\int_{\Omega}e^{\alpha\lvert u\rvert^{\frac{N}{N-1}}}\;\mathrm{d}x\leq C_{N,\alpha}\lvert\Omega\rvert
$$
for every $u\in W_{0}^{1,N}(\Omega)$ with $\|\nabla u\|_{N}\leq 1$. Moreover, $\alpha_{N}$ is the optimal constant in the sense that if $\alpha>\alpha_{N}$, the aforementioned inequality can no longer hold with some $C$ independent of $u$. The Trudinger-Moser inequality is the name given to this inequality nowadays.

It should be emphasized that  the above inequality make no sense when the domains have infinite volume,  therefore, it is necessary to look into variations of the Trudinger-Moser inequality in this situation.

Let
$$
\phi_{N}(t)=e^{t}-\sum\limits_{j=0}^{N-2}\frac{t^{j}}{j!},
$$
we provide two different versions of the Trudinger-Moser inequality in $\R^{N}$. Let $\alpha\in(0,\alpha_{N})$; notice that $\alpha=\alpha_{N}$ is not allowed; there are positive constants $C_{N}$ and $C_{N,\alpha}$ such that
\begin{align}
    &A(\alpha):=\sup\limits_{\|\nabla u\|_{N}\leq 1}\frac{1}{\|u\|_{N}^{N}}\int_{\R^{N}}\phi_{N}(\alpha\lvert u\rvert^{\frac{N}{N-1}})\;\mathrm{d}x\leq C_{N,\alpha},\label{A()}\\
    &B:=\sup\limits_{\|\nabla u\|_{N}^{N}+\|u\|_{N}^{N}\leq 1}\int_{\R^{N}}\phi_{N}(\alpha_{N}\lvert u\rvert^{\frac{N}{N-1}})\;\mathrm{d}x\leq C_{N}.\label{B()}
\end{align}
The constant $\alpha_{N}$ is sharp in the sense that $\lim_{\alpha\uparrow\alpha_{N}}A(\alpha)=\infty$, and if the constant $\alpha_{N}$ is replaced by any $\alpha>\alpha_{N}$, $B$ is infinite.

When $\Omega$ contains the origin, Adimurthi and Sandeep \cite{2} generalized Trudinger-Moser inequality to a singular version, namely, let $\alpha>0$ and $0\leq\beta<N$,
$$
\sup\limits_{u\in W_{0}^{1,N}(\Omega),\;\|\nabla u\|_{N}\leq 1}
\int_{\Omega}\frac{e^{\alpha\lvert u\rvert^{\frac{N}{N-1}}}}{\lvert x\rvert^{\beta}}\;\mathrm{d}x<+\infty\Leftrightarrow\frac{\alpha}{\alpha_{N}}+\frac{\beta}{N}\leq 1.
$$

The Trudinger-Moser inequality has been attempted to be extended to infinite volume domains by Cao \cite{5}, Ogawa \cite{25,26} in dimension two and by do  \cite{24}, Adachi and Tanaka \cite{1}, and Kozono, Sato and Wadade \cite{11} in higher dimension. It is interesting to note that the Trudinger-Moser type inequality can only be established for the subcritical case when only the seminorm $\|\nabla u\|_{N}$ is used in the restriction of the function class. Indeed, \eqref{A()} has been proved in \cite{24} and \cite{1} if $\alpha<\alpha_{N}$. Futhermore, their conclusions are actually sharp in the sense that the supremum is infinity when $\alpha\geq\alpha_{N}$. In order to achieve the critical case $\alpha=\alpha_{N}$, Ruf \cite{27} and then Li and Ruf \cite{16} need to use the full form in $W^{1,N}$, namely, $(\|u\|_{N}^{N}+\|\nabla u\|_{N}^{N})^{\frac{1}{N}}$. They also obtain that $\alpha_{N}$ is sharp without accident. It is worth mentioning that all of the studies mentioned above heavily rely on the symmetrization argument. An alternative proof of \eqref{B()} without using symmetrization has been given by Lam and Lu \cite{13}. Different proofs of \eqref{A()} and \eqref{B()} have also been given without using symmetrization in settings such as on the Heisenberg group or high and fractional order Sobolev spaces where symmetrization is not available.(see the work in Lam and Lu \cite{12} and Lam, Lu and Tang \cite{15} which extend the earlier work by Cohn and Lu \cite{6} on finite domains). These works \cite{12, 13, 15} which we just mentioned use the level sets of functions under consideration to develop arguments from local to global inequalities. Subsequntly, Masmoudi and Sani \cite{22} established the Trudinger-Moser inequalities with the exact growth condition in $\R^{N}$, these inequalities play an important role in geometric analysis and partial differential equations.

There are several other related expansions of the Trudinger-Moser inequality. One interesting direction is to generalize the Trudinger-Moser inequality to the case of anisotropic norm. In 2012, Wang and Xia \cite{32} investigated a sharp Trudinger-Moser inequality involving the anisotropic Dirichlet norm $\left(\int_{\Omega}F^{N}(\nabla u)\;\mathrm{d}x\right)^{\frac{1}{N}}$ on $W_{0}^{1,N}(\Omega)$, precisely,
$$
\sup\limits_{u\in W_{0}^{1,N}(\Omega),\;\|F(\nabla u)\|_{N}\leq 1}
\int_{\Omega}e^{\lambda\lvert u\rvert^{\frac{N}{N-1}}}\;\mathrm{d}x<+\infty\Leftrightarrow 0<\lambda\leq\lambda_{N}=N^{\frac{N}{N-1}}\kappa_{N}^{\frac{1}{N-1}}.
$$

After that, Zhou C-L and Zhou C-Q \cite{34} obtained the existence of extremal functions for anisotropic Trudinger-Moser inequality. Similar to Ruf's work \cite{27}, Zhou C-L and Zhou C-Q \cite{35} expanded results to the entire space by substituting the full anisotropic Sobolev norm $\left(\int_{\Omega}(F^{N}(\nabla u)+\lvert u\rvert^{N})\;\mathrm{d}x\right)^{\frac{1}{N}}$ for the anisotropic Dirichlet norm $\left(\int_{\Omega}F^{N}(\nabla u)\;\mathrm{d}x\right)^{\frac{1}{N}}$. Futhermore, in \cite{17}, Liu established anisotropic Trudinger-Moser inequalities associated with the exact growth in $\R^{N}$,  moreover, the author also calculated the values of the supremums of anisotropic Trudinger--Moser inequalities, and  existence and nonexistence of  the extremal functions  are also obtained in certain cases.

The main purpose of this paper is to research sharp anisotropic singular Trudinger-Moser inequalities in the Sobolev-type spaces $D^{N,q}(\R^{N}),\;q\geq 1$, namely the completion of $C_{0}^{\infty}(\R^{N})$ under norm $\|\nabla u\|_{N}+\|u\|_{q}$. We note that when $q=N$, $D^{N,q}(\R^{N})=W^{1,N}(\R^{N})$.

In this paper, we will always assume that
$$
N\geq 2,\;\;0\leq\beta<N,\;\;0\leq\lambda<\lambda_{N},\;\;q\geq 1
$$
and consider the following function:
$$
\Phi_{N,q,\beta}(t)=\left\{\begin{array}{cc}
\sum\limits_{j\in\N,\;j>\frac{q(N-1)}{N}(1-\frac{\beta}{N})}\frac{t^{j}}{j!} & {\mathrm{if}}\;\beta>0, \\
\sum\limits_{j\in\N,\;j\geq\frac{q(N-1)}{N}}\frac{t^{j}}{j!} & {\mathrm{if}}\;\beta=0.
\end{array}\right.
$$

Our primary objective in this paper is to demonstrate that:
\begin{theorem}\label{1.1}
Let $p>q(1-\frac{\beta}{N})(p\geq q\;if\;\beta=0)$. Then there exists a constant $C=C(N,p,q,\lambda,\beta)>0$ such that for all $u\in D^{N,q}(\R^{N})$, $\|F(\nabla u)\|_{N}\leq 1$, there holds
$$
\int_{\R^{N}}\frac{\exp(\lambda(1-\frac{\beta}{N})\lvert u\rvert^{\frac{N}{N-1}})\lvert u\rvert^{p}}{F^{0}(x)^{\beta}}\;\mathrm{d}x\leq C\|u\|_{q}^{q(1-\frac{\beta}{N})}.
$$
Moreover, the constant $\lambda_{N}$ is sharp in the sense that if $\lambda\geq\lambda_{N}$, the constant $C$ cannot be uniform in function u.
\end{theorem}

A consequence of Theorem \ref{1.1} is the following anisotropic subcritical Trudinger-Moser inequality:
\begin{corollary}\label{1.2}
The constant
$$
{\mathrm{ATMSC}}(N,q,\lambda,\beta)=
\sup\limits_{\|F(\nabla u)\|_{N}\leq 1}
\frac{1}{\|u\|_{q}^{q(1-\frac{\beta}{N})}}
\int_{\R^{N}}\frac{\Phi_{N,q,\beta}(\lambda(1-\frac{\beta}{N})\lvert u\rvert^{\frac{N}{N-1}})}{F^{0}(x)^{\beta}}\;\mathrm{d}x<\infty
$$
and the constant $\lambda_{N}$ is sharp.
\end{corollary}

The special case $q=N$ of Corollary \ref{1.2} is a result in \cite{28}, and we can easily draw more conclusions which are similar to those in \cite{28}.
\begin{theorem}\label{1.3}
Let $a>0$ and $b>0$. Then
\begin{align}
    {\mathrm{ATMC}}_{a,b}(N,q,\beta)=\sup\limits_{\|F(\nabla u)\|_{N}^{a}+\|u\|_{q}^{b}\leq 1}\int_{\R^{N}}\frac{\Phi_{N,q,\beta}(\lambda_{N}(1-\frac{\beta}{N})\lvert u\rvert^{\frac{N}{N-1}})}{F^{0}(x)^{\beta}}\;\mathrm{d}x<\infty\Leftrightarrow b\leq N.\notag
\end{align}
The constant $\lambda_{N}$ is sharp in the sense that if $\lambda>\lambda_{N}$, the above supremum will be infinite. Also, there are positive constants $c(N,q,\beta)$ and $C(N,q,\beta)$ such that when $\lambda\lesssim\lambda_{N}$,
$$
\frac{c(N,q,\beta)}{\left(1-(\frac{\lambda}{\lambda_{N}})^{N-1}\right)^{\frac{q(1-\frac{\beta}{N})}{N}}}\leq {\mathrm{ATMSC}}(N,q,\lambda,\beta)\leq \frac{C(N,q,\beta)}{\left(1-(\frac{\lambda}{\lambda_{N}})^{N-1}\right)^{\frac{q(1-\frac{\beta}{N})}{N}}}.
$$
Moreover, we have the following identity:
$$
{\mathrm{ATMC}}_{a,b}(N,q,\beta)=\sup\limits_{\lambda\in(0,\lambda_{N})}
\left(\frac{1-(\frac{\lambda}{\lambda_{N}})^{a\frac{N-1}{N}}}{(\frac{\lambda}{\lambda_{N}})^{b\frac{N-1}{N}}}\right)^{\frac{q}{b}(1-\frac{\beta}{N})}{\mathrm{ATMSC}}(N,q,\lambda,\beta).
$$
\end{theorem}

Then, we investigate the anisotropic singular Trudinger-Moser inequality with exact growth, which is our main objective.
\begin{theorem}\label{1.4}
Let $p\geq q\geq 1$, $d>0$. Then there exists a constant $C=C(N,p,q,\beta,d)>0$ such that for all $u\in D^{N,q}(\R^{N})$, $\|F(\nabla u)\|_{N}\leq 1$, there holds
$$
\int_{\R^{N}}\frac{\Phi_{N,q,\beta}(\lambda_{N}(1-\frac{\beta}{N})u^{\frac{N}{N-1}})}{(1+d\lvert u\rvert^{\frac{p}{N-1}(1-\frac{\beta}{N})})F^{0}(x)^{\beta}}\;\mathrm{d}x\leq C\|u\|_{q}^{q(1-\frac{\beta}{N})}.
$$
Moreover, the inequality fails if $p<q$ or the constant $\lambda_{N}$ is replaced by any $\lambda>\lambda_{N}$.
\end{theorem}

It is only logical to wonder what kinds of inequalities will hold in situation $b>N$ after studying the Trudinger-Moser inequalities under the norm constraint $\|F(\nabla u)\|_{N}^{a}+\|u\|_{q}^{b}\leq 1$ (for $b\leq N$). Can we still have a valid Trudinger-Moser type inequality, in particular, if we would want to have the norm constraint $\|F(\nabla u)\|_{N}^{a}+\|u\|_{q}^{b}\leq 1$(for $b>N$)? Indeed, we will respond to this question by the following versions of the Trudinger-Moser inequality in $D^{N,q}(\R^{N})$ which are new even when $q=N$. The isotropic versions of these inequalities were studied earlier by Ibrahim, Masmoudi and Nakanishi \cite{9}, Lam and Lu \cite{14}, Masmoudi and Sani \cite{21}, \cite{22}, Lu and Tang \cite{18}, Lu, Tang and Zhu \cite{20}. The findings presented here are enhanced anisotropic versions of the aforementioned results.

\begin{theorem}\label{1.5}
Let $a,d>0$, $k>1$, and $p\geq q\geq 1$. Then there exists a constant $C=C(N,p,q,\beta,a,d,k)>0$ such that for all $u\in D^{N,q}(\R^{N})$, $\|F(\nabla u)\|_{N}^{a}+\|u\|_{q}^{kN}\leq 1$, there holds
$$
\int_{\R^{N}}\frac{\Phi_{N,q,\beta}(\lambda_{N}(1-\frac{\beta}{N})u^{\frac{N}{N-1}})}{(1+d\lvert u\rvert^{\frac{p}{N-1}(1-\frac{1}{k})(1-\frac{\beta}{N})})F^{0}(x)^{\beta}}\;\mathrm{d}x\leq C.
$$
Moreover, the inequality fails if $p<q$ or the constant $\lambda_{N}$ is replaced by any $\lambda>\lambda_{N}$.
\end{theorem}

An equivalent version of Theorem \ref{1.5} is the following:
\begin{theorem}\label{1.6}
Let $a,d>0$, $k>1$, and $p\geq q\geq 1$. Then there exists a constant $C=C(N,p,q,\beta,a,d,k)>0$ such that for all $u\in D^{N,q}(\R^{N})$, $\|F(\nabla u)\|_{N}^{a}+\|u\|_{q}^{N}\leq 1$, there holds
$$
\int_{\R^{N}}\frac{\Phi_{N,q,\beta}(\lambda_{N}(1-\frac{\beta}{N})u^{\frac{N}{N-1}})}{(1+d\lvert u\rvert^{\frac{p}{N-1}(1-\frac{1}{k})(1-\frac{\beta}{N})})F^{0}(x)^{\beta}}\;\mathrm{d}x\leq C\|u\|_{q}^{q(1-\frac{1}{k})(1-\frac{\beta}{N})}.
$$
Moreover, the inequality fails if $p<q$ or the constant $\lambda_{N}$ is replaced by any $\lambda>\lambda_{N}$.
\end{theorem}

As consequences of our theorems, there are:
\begin{corollary}\label{1.7}
Let $k\geq 1$. There holds
\begin{align}
    &\sup\limits_{\|F(\nabla u)\|_{N}^{kN}+\|u\|_{N}^{kN}\leq 1}
    \int_{\R^{N}}\frac{\Phi_{N,\beta}(\lambda_{N}(1-\frac{\beta}{N})u^{\frac{N}{N-1}})}{(1+d\lvert u\rvert^{\frac{N}{N-1}(1-\frac{1}{k})(1-\frac{\beta}{N})})F^{0}(x)^{\beta}}\;\mathrm{d}x<\infty,\notag\\
    &\sup\limits_{\|F(\nabla u)\|_{N}^{N}+\|u\|_{N}^{N}\leq 1}
    \frac{1}{\|u\|_{N}^{N(1-\frac{1}{k})(1-\frac{\beta}{N})}}
    \int_{\R^{N}}\frac{\Phi_{N,\beta}(\lambda_{N}(1-\frac{\beta}{N})u^{\frac{N}{N-1}})}{(1+d\lvert u\rvert^{\frac{N}{N-1}(1-\frac{1}{k})(1-\frac{\beta}{N})})F^{0}(x)^{\beta}}\;\mathrm{d}x<\infty.\notag
\end{align}
\end{corollary}

Since it can be verified easily that the constant $C(N,p,q,\beta,a,k)$ in Theorem \ref{1.6} tends to a constant $C(N,p,q,\beta,a)$ as $k\rightarrow\infty$, by Fatou's lemma, we get:

\begin{corollary}\label{1.8}
Let $a,d>0$, and $p\geq q\geq 1$. Then there exists a constant $C=C(N,p,q,\beta,a,d)>0$ such that for all $u\in D^{N,q}(\R^{N})$, $\|F(\nabla u)\|_{N}^{a}+\|u\|_{q}^{N}\leq 1$, there holds
$$
\int_{\R^{N}}\frac{\Phi_{N,q,\beta}(\lambda_{N}(1-\frac{\beta}{N})u^{\frac{N}{N-1}})}{(1+d\lvert u\rvert^{\frac{p}{N-1}(1-\frac{\beta}{N})})F^{0}(x)^{\beta}}\;\mathrm{d}x\leq C\|u\|_{q}^{q(1-\frac{\beta}{N})}.
$$
Moreover, the inequality fails if $p<q$ or the constant $\lambda_{N}$ is replaced by any $\lambda>\lambda_{N}$.
\end{corollary}

The paper is structured as shown below. We will review a few lemmas which are necessary to support our theorems in Section 2. Two lemmas which are crucial in the following proofs are established in Section 3. We provide the proof for Theorem \ref{1.1} in Section 4. And we provide the proof for Theorem \ref{1.4} in Section 5. Moreover, Theorem \ref{1.5} and Theorem \ref{1.6} are established in Section 6.

\section{Preliminaries}

In this section, we provide some preliminary information we will need later. Let $F:\R^{N}\rightarrow[0,+\infty)$ be a convex function of class $C^{2}(\R^{N}\setminus\{0\})$, which is even and positively homogeneous of degree 1, this will induce
$$
F(t\xi)=\lvert t\rvert F(\xi)\;\;\;\mathrm{for}\;\mathrm{any}\;t\in\R,\;\xi\in\R^{N}.
$$

We further assume that $F(\xi)>0$ for any $\xi\neq 0$ and $Hess(F^{2})$ is positive definite in $\R^{N}\setminus\{0\}$, leading $Hess(F^{N})$ is positive definite in $\R^{N}\setminus\{0\}$ by Xie and Gong \cite{33}. For such a function $F$, there are two constants $0<a\leq b<\infty$ such that $a\lvert\xi\rvert\leq F(\xi)\leq b\lvert\xi\rvert$ for any $\xi\in\R^{N}$. A typical example is $F(\xi)=(\sum_{i}\lvert \xi_{i}\rvert^{q})^{\frac{1}{q}}$ for $q\in(1,+\infty)$. Let $F^{0}$ be the support function of $K:=\{x\in\R^{N}:F(x)\leq 1\}$, which is defined by
$$
F^{0}(x):=\sup\limits_{\xi\in K}\langle x,\xi\rangle,
$$
then $F^{0}:\R^{N}\rightarrow[0,+\infty)$ is also a convex, homogeneous function of class $C^{2}(\R^{N}\setminus\{0\})$.

From \cite{3}, $F^{0}$ is dual to $F$ in the sense that
$$
F^{0}(x)=\sup\limits_{\xi\neq 0}\frac{\langle x,\xi\rangle}{F(\xi)}, \;F(x)=\sup\limits_{\xi\neq 0}\frac{\langle x,\xi\rangle}{F^{0}(\xi)}.
$$

Consider map $\phi:S^{N-1}\rightarrow\R^{N}$, $\phi(\xi)=F_{\xi}(\xi)$. Its image $\phi(S^{N-1})$ is a smooth, convex hypersurface in $\R^{N}$, known as the Wulff shape (or equilibrium crystal shape) of $F$. As a result, $\phi(S^{N-1})=\{x\in\R^{N}|F^{0}(x)=1\}$(see \cite{31}, Proposition 1). Denote $W_{r}(x_{0})=\{x\in\R^{N}:F^{0}(x-x_{0})\leq r\}$, in particular, we call $W_{r}$ a Wulff ball with radius $r$ and center at the origin and $\kappa_{N}$ as the Lebesgue measure of $W_{1}$.

Accordingly, we provide some simple properties of $F$, as a direct consequence of the assumption on $F$, also found in \cite{4, 7, 31, 32}.
\begin{lemma}\label{2.1}
There are\\
(i) $|F(x)-F(y)|\leq F(x+y) \leq F(x)+F(y)$;\\
(ii) $\frac{1}{C}\leq |\nabla F(x)|\leq C$ and  $\frac{1}{C}\leq |\nabla F^{0}(x)|\leq C$ for some $C>0$ and any $x\neq 0$; \\
(iii) $\langle x, \nabla F(x)\rangle= F(x), \langle x, \nabla F^{0}(x)\rangle = F^{0}(x)$ for  any $x\neq 0$;\\
(iv) $\int\limits_{\partial\mathcal{W}_r}\frac{1}{|\nabla F^{0}(x)|}d\sigma= N\kappa_Nr^{N-1}$;\\
(v) $F_{\xi_i}(\xi)\xi_i=F(\xi),  F_{\xi_i}(t\xi)=sgn(t) F_{\xi_i}(\xi)$ for any $\xi \neq 0$ and $t\neq 0$;\\
(vi) $F(\nabla F^{0}(x))=1, F^0(\nabla F(\xi))=1$   for any   $x, \xi \neq 0$;\\
(vii)  $F(\xi)F^0_x(\nabla F(\xi))=\xi, F^0(x)F_\xi(\nabla F^0(x))=x$   for any   $x, \xi \neq 0$.
\end{lemma}

The convex symmetrization defined in \cite{3} will be used. It generalizes the Schwarz symmetrization(see \cite{29}). Let us consider a measurable function $u$ on $\Omega\subset\R^{N}$, one-dimensional decreasing rearrangement of $u$ is
$$
u^{\sharp}(t)=\sup\{s\geq 0:\;\lvert\{x\in\Omega:u(x)\geq s\}\rvert >t\},\;\;\;t\in\R.
$$
The convex symmetrization of $u$ with respect to $F$ is defined by
$$
u^{\star}(x)=u^{\sharp}\left(\kappa_{N}F^{0}(x)^{N}\right),\;\;\;x\in\Omega^{\star},
$$
where $\kappa_{N}F^{0}(x)^{N}$ is the Lebesgue measure of a homothetic Wulff ball with radius $F^{0}(x)$ and $\Omega^{\star}$ is a homothetic Wulff ball centered at the origin having the same measure as $\Omega$. Recalling that, the Schwarz symmetrization of $u$ is defined by
$$
u^{\ast}(x)=u^{\sharp}(\omega_{N}\lvert x\rvert^{N}),\;\;\;x\in\Omega^{\ast},
$$
where $\omega_{N}\lvert x\rvert^{N}$ is the Lebesgue measure of a Euclidean ball with radius $\lvert x\rvert$ and $\Omega^{\ast}$ is a Euclidean ball centered at the origin having the same measure as $\Omega$. As a result, it is simple to determine that the Schwarz symmetrization is the convex symmetrization for $F=\lvert\cdot\rvert$.

In addition, the convex symmetrization has the following characteristic.
\begin{lemma}\label{2.2}\rm{(P\"{o}lya-Szeg\"{o} Inequality)}
If $u\in D^{N,q}(\R^{N})$, then $u^{\star}\in D^{N,q}(\R^{N})$, and
\begin{align}
    &\int_{\R^{N}}\lvert u\rvert^{N}\;\mathrm{d}x\geq\int_{\R^{N}}\lvert u^{\star}\rvert^{N}\;\mathrm{d}x,\;\;\;\mathrm{the\;equation\;holds\;if\;and\;only\;if}\;u\geq 0,\notag\\
    &\int_{\R^{N}}F^{N}(\nabla u)\;\mathrm{d}x\geq\int_{\R^{N}}F^{N}(\nabla u^{\star})\;\mathrm{d}x.\notag
\end{align}
\end{lemma}
\begin{lemma}\label{2.3}\rm{(Hardy-Littlewood Inequality)}
Let $f$ and $g$ be nonnegative functions on $\R^{N}$, vanishing at infinity. Then
$$
\int_{\R^{N}}f(x)g(x)\;\mathrm{d}x\leq \int_{\R^{N}}f^{\star}(x)g^{\star}(x)\;\mathrm{d}x.
$$
\end{lemma}

One can easily arrive at a variant of Lemma 2.2 in \cite{19}.
\begin{lemma}\label{2.4}
Given any sequence $a=\{a_{k}\}_{k\geq 0}$, let
$$
\|a\|_{1}=\sum\limits_{k=0}^{\infty}\lvert a_{k}\rvert,\;\;\;
\|a\|_{N}=\left(\sum\limits_{k=0}^{\infty}\lvert a_{k}\rvert^{N}\right)^{\frac{1}{N}},\;\;\;
\|a\|_{e}=\left(\sum\limits_{k=0}^{\infty}\lvert a_{k}\rvert^{q}e^{k}\right)^{\frac{1}{q}}
$$
and
$$
\mu(h)=\inf\{\|a\|_{e}:\;\|a\|_{1}=h,\|a\|_{N}\leq 1\}.
$$
Then for $h>1$,
$$
\mu(h)\sim\frac{\exp(\frac{h^{\frac{N}{N-1}}}{q})}{h^{\frac{1}{N-1}}}.
$$
\end{lemma}

Finally, we can infer from \cite{8} that
\begin{lemma}\label{2.5}
Let $N\geq 2$, $0<t<N$. Then the embedding
$$
W_{rad}^{1,N}(\R^{N})\cap L^{q}(\R^{N})\hookrightarrow L^{r}(\R^{N},\lvert x\rvert^{-t})
$$
is compact for all $r\geq q$. Moreover, the embedding
$$
W_{rad}^{1,N}(\R^{N})\cap L^{q}(\R^{N})\hookrightarrow L^{r}(\R^{N})
$$
is continuous for all $r\geq q$ and compact for all $r>q$.
\end{lemma}

\section{Two Lemmas}

In this section, we introduce two lemmas which are extremely important to the following proofs.
\begin{lemma}\label{3.1}
Denote $\gamma=(\frac{\kappa_{N}}{\omega_{N}})^{\frac{1}{N}}$. For every $u\in D^{N,q}(\R^{N})$,
\begin{align}
    &\int_{\R^{N}}F(\nabla u^{\star})^{N}\;\mathrm{d}x=\gamma^{N}\int_{\R^{N}}\lvert\nabla u^{\ast}\rvert^{N}\;\mathrm{d}x,\notag\\
    &\int_{\R^{N}}\frac{\exp(\lambda(1-\frac{\beta}{N})\lvert u^{\star}\rvert^{\frac{N}{N-1}})\lvert u^{\star}\rvert^{p}}{F^{0}(x)^{\beta}}\;\mathrm{d}x
    =\gamma^{\beta}\int_{\R^{N}}\frac{\exp(\lambda(1-\frac{\beta}{N})\lvert u^{\ast}\rvert^{\frac{N}{N-1}})\lvert u^{\ast}\rvert^{p}}{\lvert x\rvert^{\beta}}\;\mathrm{d}x.\notag
\end{align}
\end{lemma}

\begin{proof}
By Lemma \ref{2.1},
\begin{align}
    &\int_{\R^{N}}F(\nabla u^{\star})^{N}\;\mathrm{d}x\notag\\
    =&\int_{0}^{\infty}F(u^{\sharp'}(\kappa_{N}F^{0}(x)^{N})\kappa_{N}NF^{0}(x)^{N-1}\nabla F^{0}(x))^{N}\int_{\partial W_{r}}\frac{1}{\lvert\nabla F^{0}(x)\rvert}\;\mathrm{d}\sigma\mathrm{d}r\notag\\
    =&N\kappa_{N}\int_{0}^{\infty}\lvert u^{\sharp'}(\kappa_{N}r^{N})\kappa_{N}Nr^{N-1}\rvert^{N}r^{N-1}\;\mathrm{d}r\;\;\;(\mathrm{setting}\;r=\gamma^{-1}t)\notag\\
    =&\gamma^{N}N\omega_{N}\int_{0}^{\infty}\lvert u^{\sharp'}(\omega_{N}t^{N})\omega_{N}Nt^{N-1}\rvert^{N}t^{N-1}\;\mathrm{d}t\notag\\
    =&\gamma^{N}\int_{0}^{\infty}\lvert u^{\sharp'}(\omega_{N}t^{N})\omega_{N}Nt^{N-1}\rvert^{N}\int_{\partial B_{t}}1\;\mathrm{d}\sigma\mathrm{d}t\notag\\
    =&\gamma^{N}\int_{\R^{N}}\lvert\nabla u^{\ast}\rvert^{N}\;\mathrm{d}x\notag
\end{align}
and
\begin{align}
    &\int_{\R^{N}}\frac{\exp(\lambda(1-\frac{\beta}{N})\lvert u^{\star}\rvert^{\frac{N}{N-1}})\lvert u^{\star}\rvert^{p}}{F^{0}(x)^{\beta}}\;\mathrm{d}x\notag\\
    =&\int_{0}^{\infty}\frac{\exp(\lambda(1-\frac{\beta}{N})\lvert u^{\sharp}(\kappa_{N}F^{0}(x)^{N})\rvert^{\frac{N}{N-1}})\lvert u^{\sharp}(\kappa_{N}F^{0}(x)^{N})\rvert^{p}}{F^{0}(x)^{\beta}}\int_{\partial W_{r}}\frac{1}{\lvert\nabla F^{0}(x)\rvert}\;\mathrm{d}\sigma\mathrm{d}r\notag\\
    =&N\kappa_{N}\int_{0}^{\infty}\frac{\exp(\lambda(1-\frac{\beta}{N})\lvert u^{\sharp}(\kappa_{N}r^{N})\rvert^{\frac{N}{N-1}})\lvert u^{\sharp}(\kappa_{N}r^{N})\rvert^{p}}{r^{\beta}}r^{N-1}\;\mathrm{d}r\;\;\;(\mathrm{setting}\;r=\gamma^{-1}t)\notag\\
    =&\gamma^{\beta}N\omega_{N}\int_{0}^{\infty}\frac{\exp(\lambda(1-\frac{\beta}{N})\lvert u^{\sharp}(\omega_{N}t^{N})\rvert^{\frac{N}{N-1}})\lvert u^{\sharp}(\omega_{N}t^{N})\rvert^{p}}{t^{\beta}}t^{N-1}\;\mathrm{d}t\notag\\
    =&\gamma^{\beta}\int_{\R^{N}}\frac{\exp(\lambda(1-\frac{\beta}{N})\lvert u^{\ast}\rvert^{\frac{N}{N-1}})\lvert u^{\ast}\rvert^{p}}{\lvert x\rvert^{\beta}}\;\mathrm{d}x.\notag
\end{align}
\end{proof}

\begin{lemma}\label{3.2}
In $D^{N,q}(\R^{N})$, for any nonnegative decreasing and Wulff symmetric function $u(r)$ with $r=F^{0}(x)$. If $u(R)>\frac{1}{N}\left(\frac{K}{\kappa_{N}}\right)^{\frac{1}{N}}$ and $\int_{\R^{N}\setminus W_{R}}F^{N}(\nabla u)\;\mathrm{d}x\leq K$ for some $R,K>0$, then
$$
\frac{\exp(\lambda_{N}K^{\frac{1}{1-N}}u(R)^{\frac{N}{N-1}})}{u(R)^{\frac{q}{N-1}}}R^{N}\leq C(N,q)\frac{\int_{R}^{\infty}\lvert u(r)\rvert^{q}r^{N-1}\;\mathrm{d}r}{K^{\frac{q}{N-1}}}.
$$
In particular, for $F=\lvert\cdot\rvert$, if $u(R)>1$ and $\int_{\R^{N}\setminus W_{R}}\lvert\nabla u\rvert^{N}\;\mathrm{d}x\leq 1$, there is
$$
\frac{\exp(\alpha_{N}u(R)^{\frac{N}{N-1}})}{u(R)^{\frac{q}{N-1}}}R^{N}\leq C(N,q)\int_{R}^{\infty}\lvert u(r)\rvert^{q}r^{N-1}\;\mathrm{d}r.
$$
\end{lemma}
\begin{proof}
Denote
$$
h_{k}=N\left(\frac{\kappa_{N}}{K}\right)^{\frac{1}{N}}u(Re^{\frac{k}{N}}),\;\;\;a_{k}=h_{k}-h_{k+1}\geq 0,
$$
then
$$
\|a\|_{1}=h_{0}=N\left(\frac{\kappa_{N}}{K}\right)^{\frac{1}{N}}u(R)>1.
$$
By H\"{o}lder's inequality,
\begin{align}
    a_{k}=h_{k}-h_{k+1}&=N\left(\frac{\kappa_{N}}{K}\right)^{\frac{1}{N}}(u(Re^{\frac{k}{N}})-u(Re^{\frac{k+1}{N}}))\notag\\
    &=N\left(\frac{\kappa_{N}}{K}\right)^{\frac{1}{N}}\int_{Re^{\frac{k+1}{N}}}^{Re^{\frac{k}{N}}}u'(r)\;\mathrm{d}r\notag\\
    &\leq N\left(\frac{\kappa_{N}}{K}\right)^{\frac{1}{N}}\left(\int_{Re^{\frac{k}{N}}}^{Re^{\frac{k+1}{N}}}\lvert u'(r)\rvert^{N}r^{N-1}\;\mathrm{d}r\right)^{\frac{1}{N}}\left(\int_{Re^{\frac{k}{N}}}^{Re^{\frac{k+1}{N}}}\frac{1}{r}\;\mathrm{d}r\right)^{\frac{N-1}{N}}\notag\\
    &=(\frac{N\kappa_{N}}{K})^{\frac{1}{N}}\left(\int_{Re^{\frac{k}{N}}}^{Re^{\frac{k+1}{N}}}\lvert u'(r)\rvert^{N}r^{N-1}\;\mathrm{d}r\right)^{\frac{1}{N}}.\notag
\end{align}
After this, by Lemma \ref{2.1},
\begin{align}
    \|a\|_{N}^{N}=\sum\limits_{k=0}^{\infty}\lvert a_{k}\rvert^{N}&\leq\frac{N\kappa_{N}}{K}\int_{R}^{\infty}\lvert u'(r)\rvert^{N}r^{N-1}\;\mathrm{d}r\notag\\
    &=\frac{N\kappa_{N}}{K}\int_{R}^{\infty}F^{N}(u'(r)\nabla F^{0}(x))r^{N-1}\;\mathrm{d}r\notag\\
    &=\frac{1}{K}\int_{R}^{\infty}F^{N}(u'(r)\nabla F^{0}(x))\int_{\partial W_{r}}\frac{1}{\lvert\nabla F^{0}(x)\rvert}\;\mathrm{d}\sigma\mathrm{d}r\notag\\
    &=\frac{1}{K}\int_{R^{N}\setminus W_{R}}F^{N}(\nabla u)\;\mathrm{d}x\leq 1.\notag
\end{align}
On the other hand,
\begin{align}
    \int_{R}^{\infty}\lvert u(r)\rvert^{q}r^{N-1}\;\mathrm{d}r&=\sum\limits_{k=0}^{\infty}\int_{Re^{\frac{k}{N}}}^{Re^{\frac{k+1}{N}}}\lvert u(r)\rvert^{q}r^{N-1}\;\mathrm{d}r\notag\\
    &\geq\sum\limits_{k=0}^{\infty}\lvert u(Re^{\frac{k+1}{N}})\rvert^{q}\int_{Re^{\frac{k}{N}}}^{Re^{\frac{k+1}{N}}}r^{N-1}\;\mathrm{d}r\notag\\
    &=\sum\limits_{k=0}^{\infty}\lvert u(Re^{\frac{k+1}{N}})\rvert^{q}\frac{R^{N}e^{k+1}(1-e^{-1})}{N}\notag\\
    &=\frac{R^{N}(1-e^{-1})}{N}\sum\limits_{k=0}^{\infty}\lvert u(Re^{\frac{k+1}{N}})\rvert^{q}e^{k+1}\notag\\
    &=\frac{R^{N}(1-e^{-1})}{N}\sum\limits_{k=1}^{\infty}\lvert u(Re^{\frac{k}{N}})\rvert^{q}e^{k}\notag\\
    &=\frac{R^{N}(1-e^{-1})}{N^{q+1}}\left(\frac{K}{\kappa_{N}}\right)^{\frac{q}{N}}\sum\limits_{k=1}^{\infty}\lvert h_{k}\rvert^{q}e^{k}\notag\\
    &\geq\frac{R^{N}(1-e^{-1})}{N^{q+1}}\left(\frac{K}{\kappa_{N}}\right)^{\frac{q}{N}}\sum\limits_{k=1}^{\infty}\lvert a_{k}\rvert^{q}e^{k},\notag
\end{align}
which implies
\begin{align}\label{3.1式}
    \|a\|_{e}^{q}=a_{0}^{q}+\sum\limits_{k=1}^{\infty}\lvert a_{k}\rvert^{q}e^{k}\leq h_{0}^{q}+\frac{C(N,q)}{R^{N}K^{\frac{q}{N}}}\int_{R}^{\infty}\lvert u(r)\rvert^{q}r^{N-1}\;\mathrm{d}r.
\end{align}
Now, for $R<s<Re^{\frac{2^{\frac{N}{1-N}}}{N}}:=\widetilde{R}$,
\begin{align}
    h_{0}-N\left(\frac{\kappa_{N}}{K}\right)^{\frac{1}{N}}u(s)&=N\left(\frac{\kappa_{N}}{K}\right)^{\frac{1}{N}}\int_{s}^{R}u'(r)\;\mathrm{d}r\notag\\
    &\leq N\left(\frac{\kappa_{N}}{K}\right)^{\frac{1}{N}}\left(\int_{R}^{s}\lvert u'(r)\rvert^{N}r^{N-1}\;\mathrm{d}r\right)^{\frac{1}{N}}\left(\int_{R}^{s}\frac{1}{r}\;\mathrm{d}r\right)^{\frac{N-1}{N}}\notag\\
    &\leq\frac{1}{2K^{\frac{1}{N}}}\left(\int_{R}^{\infty}F^{N}(u'(r)\nabla F^{0}(x))\int_{\partial W_{r}}\frac{1}{\lvert\nabla F^{0}(x)\rvert}\;\mathrm{d}\sigma\mathrm{d}r\right)^{\frac{1}{N}}\notag\\
    &\leq\frac{1}{2K^{\frac{1}{N}}}\left(\int_{R^{N}\setminus W_{R}}F^{N}(\nabla u)\;\mathrm{d}x\right)^{\frac{1}{N}}\notag\\
    &\leq\frac{1}{2}<\frac{h_{0}}{2}.\notag
\end{align}
Hence, $h_{0}\lesssim K^{-\frac{1}{N}}u(s)$ for $0<s<\widetilde{R}$, and
$$
\int_{R}^{\infty}\lvert u(r)\rvert^{q}r^{N-1}\;\mathrm{d}r\geq\int_{R}^{\widetilde{R}}\lvert u(r)\rvert^{q}r^{N-1}\;\mathrm{d}r\gtrsim R^{N}K^{\frac{q}{N}}h_{0}^{q}.
$$
Furthermore, by \eqref{3.1式} and Lemma \ref{2.4},
\begin{align}
    \int_{R}^{\infty}\lvert u(r)\rvert^{q}r^{N-1}\;\mathrm{d}r&\geq C(N,q)R^{N}K^{\frac{q}{N}}\|a\|_{e}^{q}\notag\\
    &\geq C(N,q)R^{N}K^{\frac{q}{N}}\frac{e^{h_{0}^{\frac{N}{N-1}}}}{h_{0}^{\frac{q}{N-1}}}\notag\\
    &=C(N,q)R^{N}K^{\frac{q}{N-1}}\frac{\exp(\lambda_{N}K^{\frac{1}{1-N}}u(R)^{\frac{N}{N-1}})}{u(R)^{\frac{q}{N-1}}}.\notag
\end{align}
For $F=\lvert\cdot\rvert$, $u(R)>1$ and $\int_{\R^{N}\setminus W_{R}}\lvert\nabla u\rvert^{N}\;\mathrm{d}x\leq 1$, we see
$$
h_{0}=N\omega_{N}^{\frac{1}{N}}>1.
$$
To conclude, we can repeat the previous processes.
\end{proof}

\section{Anisotropic Subcritical Trudinger-Moser Inequality}

In this section, we demonstrate Theorem \ref{1.1}. It illustrates an anisotropic subcritical Trudinger-Moser inequality.
\begin{proof}
It is sufficient to prove Theorem \ref{1.1} for $u\geq 0$. Since $\left(\frac{1}{F^{0}(x)^{\beta}}\right)^{\star}=\frac{1}{F^{0}(x)^{\beta}}$, by Lemma \ref{2.2}, Lemma \ref{2.3} there is
\begin{align}\label{4.1}
    &\sup\limits_{u\in D^{N,q}(\R^{N}),\;\|F(\nabla u)\|_{N}\leq 1}\frac{1}{\|u\|_{q}^{q(1-\frac{\beta}{N})}}\int_{\R^{N}}\frac{\exp(\lambda(1-\frac{\beta}{N})\lvert u\rvert^{\frac{N}{N-1}})\lvert u\rvert^{p}}{F^{0}(x)^{\beta}}\;\mathrm{d}x\notag\\
    \leq&\sup\limits_{u^{\star}\in D^{N,q}(\R^{N}),\;\|F(\nabla u^{\star})\|_{N}\leq 1}\frac{1}{\|u^{\star}\|_{q}^{q(1-\frac{\beta}{N})}}\int_{\R^{N}}\frac{\exp(\lambda(1-\frac{\beta}{N})\lvert u^{\star}\rvert^{\frac{N}{N-1}})\lvert u^{\star}\rvert^{p}}{F^{0}(x)^{\beta}}\;\mathrm{d}x.
\end{align}
Seeing that the opposite inequality of \eqref{4.1} is clear, we conclude
\begin{align}\label{4.2}
    &\sup\limits_{u\in D^{N,q}(\R^{N}),\;\|F(\nabla u)\|_{N}\leq 1}\frac{1}{\|u\|_{q}^{q(1-\frac{\beta}{N})}}\int_{\R^{N}}\frac{\exp(\lambda(1-\frac{\beta}{N})\lvert u\rvert^{\frac{N}{N-1}})\lvert u\rvert^{p}}{F^{0}(x)^{\beta}}\;\mathrm{d}x\notag\\
    =&\sup\limits_{u^{\star}\in D^{N,q}(\R^{N}),\;\|F(\nabla u^{\star})\|_{N}\leq 1}\frac{1}{\|u^{\star}\|_{q}^{q(1-\frac{\beta}{N})}}\int_{\R^{N}}\frac{\exp(\lambda(1-\frac{\beta}{N})\lvert u^{\star}\rvert^{\frac{N}{N-1}})\lvert u^{\star}\rvert^{p}}{F^{0}(x)^{\beta}}\;\mathrm{d}x.
\end{align}
After that, by Lemma \ref{3.1},
\begin{align}\label{4.3}
    &\sup\limits_{u^{\star}\in D^{N,q}(\R^{N}),\;\|F(\nabla u^{\star})\|_{N}\leq 1}\frac{1}{\|u^{\star}\|_{q}^{q(1-\frac{\beta}{N})}}\int_{\R^{N}}\frac{\exp(\lambda(1-\frac{\beta}{N})\lvert u^{\star}\rvert^{\frac{N}{N-1}})\lvert u^{\star}\rvert^{p}}{F^{0}(x)^{\beta}}\;\mathrm{d}x\notag\\
    =&\sup\limits_{u^{\ast}\in D^{N,q}(\R^{N}),\;\gamma\|\nabla u^{\ast})\|_{N}\leq 1}\frac{\gamma^{\beta}}{\|u^{\ast}\|_{q}^{q(1-\frac{\beta}{N})}}\int_{\R^{N}}\frac{\exp(\lambda(1-\frac{\beta}{N})\lvert u^{\ast}\rvert^{\frac{N}{N-1}})\lvert u^{\ast}\rvert^{p}}{\lvert x\rvert^{\beta}}\;\mathrm{d}x\notag\\
    =&\sup\limits_{u^{\ast}\in D^{N,q}(\R^{N}),\;\|\nabla u^{\ast})\|_{N}\leq 1}\frac{\gamma^{\beta+q(1-\frac{\beta}{N})-p}}{\|u^{\ast}\|_{q}^{q(1-\frac{\beta}{N})}}\int_{\R^{N}}\frac{\exp(\lambda(1-\frac{\beta}{N})\gamma^{-\frac{N}{N-1}}\lvert u^{\ast}\rvert^{\frac{N}{N-1}})\lvert u^{\ast}\rvert^{p}}{\lvert x\rvert^{\beta}}\;\mathrm{d}x.
\end{align}
We may suppose that $u\in C_{0}^{\infty}(\R^{N})$ is a nonnegative, radially decreasing function and $\|\nabla u\|_{N}\leq 1$ by using the density argument. Writen as
$$
\int_{\R^{N}}\frac{\exp(\lambda(1-\frac{\beta}{N})\gamma^{-\frac{N}{N-1}}\lvert u\rvert^{\frac{N}{N-1}})\lvert u\rvert^{p}}{\lvert x\rvert^{\beta}}\;\mathrm{d}x=I_{1}+I_{2},
$$
where
\begin{align}
    &I_{1}=\int_{\{u>1\}}\frac{\exp(\lambda(1-\frac{\beta}{N})\gamma^{-\frac{N}{N-1}}\lvert u\rvert^{\frac{N}{N-1}})\lvert u\rvert^{p}}{\lvert x\rvert^{\beta}}\;\mathrm{d}x,\notag\\
    &I_{2}=\int_{\{u\leq 1\}}\frac{\exp(\lambda(1-\frac{\beta}{N})\gamma^{-\frac{N}{N-1}}\lvert u\rvert^{\frac{N}{N-1}})\lvert u\rvert^{p}}{\lvert x\rvert^{\beta}}\;\mathrm{d}x.\notag
\end{align}
First, by Lemma \ref{2.5},
\begin{align}\label{4.4}
    I_{2}\leq e^{\lambda\gamma^{-\frac{N}{N-1}}}\|\lvert x\rvert^{-\frac{\beta}{p}}u\|_{p}^{p}\leq C(N,p,q,\lambda,\beta)\|u\|_{q}^{q(1-\frac{\beta}{N})}.
\end{align}
Next, set
$$
v=u-1\;\;\mathrm{on}\;\;\Omega(u):=\{u>1\},
$$
so that
$$
v\in W_{0}^{1,N}(\Omega(u))\;\;\mathrm{and}\;\;\|\nabla v\|_{N}\leq 1.
$$
By $(1+t^{\frac{N-1}{N}})^{\frac{N}{N-1}}\leq(1+\epsilon)t+C_{\epsilon}$ for any $\epsilon>0$ and $t\geq 0$,
\begin{align}
    I_{1}&=\int_{\Omega(u)}\frac{\exp(\lambda(1-\frac{\beta}{N})\gamma^{-\frac{N}{N-1}}\lvert v+1\rvert^{\frac{N}{N-1}}+p\ln(v+1))}{\lvert x\rvert^{\beta}}\;\mathrm{d}x\notag\\
    &\leq\int_{\Omega(u)}\frac{\exp(\lambda(1-\frac{\beta}{N})\gamma^{-\frac{N}{N-1}}\lvert v+1\rvert^{\frac{N}{N-1}}+pv)}{\lvert x\rvert^{\beta}}\;\mathrm{d}x\notag\\
    &\leq\int_{\Omega(u)}\frac{\exp(\lambda(1-\frac{\beta}{N})\gamma^{-\frac{N}{N-1}}((1+\epsilon)v^{\frac{N}{N-1}}+C_{\epsilon})+C_{\epsilon})}{\lvert x\rvert^{\beta}}\;\mathrm{d}x\notag\\
    &=C(N,p,q,\lambda,\beta)\int_{\Omega(u)}\frac{\exp(\alpha_{N}(1-\frac{\beta}{N})\lvert v\rvert^{\frac{N}{N-1}})}{\lvert x\rvert^{\beta}}\;\mathrm{d}x\notag
\end{align}
for $\epsilon=\frac{\lambda_{N}}{\lambda}-1$ and for $C_{\epsilon}$ sufficiently large. Hence, by the singular Trudinger-Moser inequality in \cite{2},
$$
I_{1}\leq C(N,p,q,\lambda,\beta)\lvert\Omega(u)\rvert^{1-\frac{\beta}{N}}.
$$
Meanwhile, it follows naturally that
$$
\lvert\Omega(u)\rvert=\int_{\Omega(u)}1\;\mathrm{d}x\leq\int_{\Omega(u)}u^{q}\;\mathrm{d}x=\|u\|_{q}^{q}.
$$
So,
\begin{align}\label{4.5}
    I_{1}\leq C(N,p,q,\lambda,\beta)\|u\|_{q}^{q(1-\frac{\beta}{N})}.
\end{align}
By \eqref{4.2}-\eqref{4.5},
$$
\sup\limits_{u\in D^{N,q}(\R^{N}),\;\|F(\nabla u)\|_{N}\leq 1}\frac{1}{\|u\|_{q}^{q(1-\frac{\beta}{N})}}\int_{\R^{N}}\frac{\exp(\lambda(1-\frac{\beta}{N})\lvert u\rvert^{\frac{N}{N-1}})\lvert u\rvert^{p}}{F^{0}(x)^{\beta}}\;\mathrm{d}x\leq C(N,p,q,\lambda,\beta).
$$

Now, we prove that the constant $\lambda_{N}$ is sharp.

Choose $\{u_{n}\}_{n=1}^{\infty}$ as follows:
$$
u_{n}(x)=\left\{\begin{array}{cc}
\left(\frac{1}{N\kappa_{N}}\right)^{\frac{1}{N}}\left(\frac{n}{N-\beta}\right)^{\frac{N-1}{N}}, & 0\leq F^{0}(x)\leq e^{-\frac{n}{N-\beta}}, \\
\left(\frac{N-\beta}{nN\kappa_{N}}\right)^{\frac{1}{N}}\ln\left(\frac{1}{F^{0}(x)}\right), & e^{-\frac{n}{N-\beta}}\leq F^{0}(x)\leq 1, \\
0, & F^{0}(x)\geq 1.
\end{array}\right.
$$
Then
$$
\int_{\R^{N}}F^{N}(\nabla u_{n})\;\mathrm{d}x=1,\;\;\;\|u_{n}\|_{q}=O(\left(\frac{1}{n}\right)^{\frac{1}{N}}).
$$
There is
\begin{align}
    &\int_{\R^{N}}\frac{\exp(\lambda_{N}(1-\frac{\beta}{N})\lvert u_{n}\rvert^{\frac{N}{N-1}})\lvert u_{n}\rvert^{p}}{F^{0}(x)^{\beta}}\;\mathrm{d}x\notag\\
    \geq&\exp\left(\lambda_{N}(1-\frac{\beta}{N})(\frac{1}{N\kappa_{N}})^{\frac{1}{N-1}}(\frac{n}{N-\beta})\right)(\frac{1}{N\kappa_{N}})^{\frac{p}{N}}(\frac{n}{N-\beta})^{\frac{p(N-1)}{N}}\int_{0}^{e^{-\frac{n}{N-\beta}}}N\kappa_{N}r^{N-1-\beta}\;\mathrm{d}r\notag\\
    =&(\frac{1}{N\kappa_{N}})^{\frac{p}{N}-1}(\frac{n}{N-\beta})^{\frac{p(N-1)}{N}}\frac{1}{N-\beta}\notag\\
    \sim&n^{\frac{p(N-1)}{N}},\notag
\end{align}
Hence,
\begin{align}
    \frac{1}{\|u_{n}\|_{q}^{q(1-\frac{\beta}{N})}}\int_{\R^{N}}\frac{\exp(\lambda_{N}(1-\frac{\beta}{N})\lvert u_{n}\rvert^{\frac{N}{N-1}})\lvert u_{n}\rvert^{p}}{F^{0}(x)^{\beta}}\;\mathrm{d}x\rightarrow\infty.\notag
\end{align}
This finishes the proof of Theorem \ref{1.1}.
\end{proof}

\section{Anisotropic Singular Trudinger-Moser Inequality with Exact Growth}

In this section, we demonstrate Theorem \ref{1.4}. It illustrates an anisotropic singular Trudinger-Moser inequality with exact growth.
\begin{proof}
In view of \eqref{4.1}-\eqref{4.3},
\begin{align}\label{5.1}
    &\sup\limits_{u\in D^{N,q}(\R^{N}),\;\|F(\nabla u)\|_{N}\leq 1}\frac{1}{\|u\|_{q}^{q(1-\frac{\beta}{N})}}\int_{\R^{N}}\frac{\Phi_{N,q,\beta}(\lambda_{N}(1-\frac{\beta}{N})\lvert u\rvert^{\frac{N}{N-1}})}{(1+d\lvert u\rvert^{\frac{p}{N-1}(1-\frac{\beta}{N})})F^{0}(x)^{\beta}}\;\mathrm{d}x\notag\\
    =&\sup\limits_{u^{\ast}\in D^{N,q}(\R^{N}),\;\|\nabla u^{\ast}\|_{N}\leq 1}\frac{\gamma^{\beta+q(1-\frac{\beta}{N})}}{\|u^{\ast}\|_{q}^{q(1-\frac{\beta}{N})}}\int_{\R^{N}}\frac{\Phi_{N,q,\beta}(\alpha_{N}(1-\frac{\beta}{N})\lvert u^{\ast}\rvert^{\frac{N}{N-1}})}{(1+d\gamma^{-\frac{p}{N-1}(1-\frac{\beta}{N})}\lvert u^{\ast}\rvert^{\frac{p}{N-1}(1-\frac{\beta}{N})})\lvert x\rvert^{\beta}}\;\mathrm{d}x.
\end{align}
We may suppose that $u\in C_{0}^{\infty}(\R^{N})$ is a nonnegative, radially decreasing function and $\|\nabla u\|_{N}\leq 1$ by using the density argument. Let $R_{1}=R_{1}(u)$ be such that
\begin{align}
    &\int_{B_{R_{1}}}\lvert\nabla u\rvert^{N}\;\mathrm{d}x=N\omega_{N}\int_{0}^{R_{1}}\lvert u_{r}\rvert^{N}r^{N-1}\;\mathrm{d}r\leq 1-\varepsilon_{0},\label{5.2}\\
    &\int_{\R^{N}\setminus B_{R_{1}}}\lvert\nabla u\rvert^{N}\;\mathrm{d}x=N\omega_{N}\int_{R_{1}}^{\infty}\lvert u_{r}\rvert^{N}r^{N-1}\;\mathrm{d}r\leq\varepsilon_{0}.\label{5.3}
\end{align}
Here $\varepsilon_{0}\in(0,1)$ is fixed and does not depend on $u$.

By  H\"{o}lder's inequality,
\begin{align}\label{5.4}
    u(r_{1})-u(r_{2})&=\int_{r_{1}}^{r_{2}}-u_{r}\;\mathrm{d}r\leq\left(\int_{r_{1}}^{r_{2}}\lvert u_{r}\rvert^{N}r^{N-1}\;\mathrm{d}r\right)^{\frac{1}{N}}\left(\ln\frac{r_{2}}{r_{1}}\right)^{\frac{N-1}{N}}\notag\\
    &\leq\left(\frac{1-\varepsilon_{0}}{N\omega_{N}}\right)^{\frac{1}{N}}\left(\ln\frac{r_{2}}{r_{1}}\right)^{\frac{N-1}{N}}\;\;\;\mathrm{for}\;0<r_{1}\leq r_{2}\leq R_{1}
\end{align}
and
\begin{align}\label{5.5}
    u(r_{1})-u(r_{2})\leq\left(\frac{\varepsilon_{0}}{N\omega_{N}}\right)^{\frac{1}{N}}\left(\ln\frac{r_{2}}{r_{1}}\right)^{\frac{N-1}{N}}\;\;\;\mathrm{for}\;R_{1}\leq r_{1}\leq r_{2}.
\end{align}
Define $R_{0}:=\inf\{r>0:\;u(r)\leq1\}\in[0,\infty)$. Hence $u(s)\leq 1$ when $s\geq R_{0}$. Without loss of generelity, we assume $R_{0}>0$.

Writen as
$$
\int_{\R^{N}}\frac{\Phi_{N,q,\beta}(\alpha_{N}(1-\frac{\beta}{N})\lvert u\rvert^{\frac{N}{N-1}})}{(1+d\gamma^{-\frac{p}{N-1}(1-\frac{\beta}{N})}\lvert u\rvert^{\frac{p}{N-1}(1-\frac{\beta}{N})})\lvert x\rvert^{\beta}}\;\mathrm{d}x=I+J,
$$
where
\begin{align}
    &I=\int_{B_{R_{0}}}\frac{\Phi_{N,q,\beta}(\alpha_{N}(1-\frac{\beta}{N})\lvert u\rvert^{\frac{N}{N-1}})}{(1+d\gamma^{-\frac{p}{N-1}(1-\frac{\beta}{N})}\lvert u\rvert^{\frac{p}{N-1}(1-\frac{\beta}{N})})\lvert x\rvert^{\beta}}\;\mathrm{d}x\notag\\
    &J=\int_{\R^{N}\setminus B_{R_{0}}}\frac{\Phi_{N,q,\beta}(\alpha_{N}(1-\frac{\beta}{N})\lvert u\rvert^{\frac{N}{N-1}})}{(1+d\gamma^{-\frac{p}{N-1}(1-\frac{\beta}{N})}\lvert u\rvert^{\frac{p}{N-1}(1-\frac{\beta}{N})})\lvert x\rvert^{\beta}}\;\mathrm{d}x.\notag
\end{align}
First, we estimate $J$. Since $u\leq1$ on $\R^{N}\setminus B_{R_{0}}$, by Lagrange remainder term of Taylor's formula and Lemma \ref{2.5} we see if $\beta>0$,
\begin{align}\label{5.6}
    J\leq C(N,q,\beta)\int_{\R^{N}\setminus B_{R_{0}}}\frac{\lvert u\rvert^{\frac{N}{N-1}(\lfloor q\frac{N-1}{N}(1-\frac{\beta}{N})\rfloor+1)}}{\lvert x\rvert^{\beta}}\;\mathrm{d}x\leq C(N,q,\beta)\|u\|_{q}^{q(1-\frac{\beta}{N})}.
\end{align}
Similar to the situation $\beta=0$,
\begin{align}\label{5.7}
    J\leq C(N,q)\|u\|_{q}^{q(1-\frac{\beta}{N})}.
\end{align}
Therefore, we just need to deal with the integral $I$.

{\bfseries Case 1:} $0<R_{0}\leq R_{1}$. By using \eqref{5.4}, for $0<r\leq R_{0}$,
$$
u(r)\leq 1+\left(\frac{1-\varepsilon_{0}}{N\omega_{N}}\right)^{\frac{1}{N}}\left(\ln\frac{R_{0}}{r}\right)^{\frac{N-1}{N}}.
$$
Since
$$
(a+b)^{t}\leq(1+\epsilon)^{\frac{t-1}{t}}a^{t}+(1-(1+\epsilon)^{-\frac{1}{t}})^{1-t}b^{t}
$$
for any $t\geq 1$, $\epsilon>0$ and $a,b\geq 0$,
$$
u^{\frac{N}{N-1}}(r)\leq(1+\epsilon)^{\frac{1}{N}}\left(\frac{1-\varepsilon_{0}}{N\omega_{N}}\right)^{\frac{1}{N-1}}\ln\frac{R_{0}}{r}+C(N,\epsilon).
$$
After that,
\begin{align}
    I&\leq\int_{B_{R_{0}}}\frac{\exp(\alpha_{N}(1-\frac{\beta}{N})\lvert u\rvert^{\frac{N}{N-1}})}{\lvert x\rvert^{\beta}}\;\mathrm{d}x\notag\\
    &\leq\int_{B_{R_{0}}}\frac{\exp(\alpha_{N}(1-\frac{\beta}{N})(1+\epsilon)^{\frac{1}{N}}\left(\frac{1-\varepsilon_{0}}{N\omega_{N}}\right)^{\frac{1}{N-1}}\ln\frac{R_{0}}{r}+\alpha_{N}(1-\frac{\beta}{N})C(N,\epsilon))}{\lvert x\rvert^{\beta}}\;\mathrm{d}x\notag\\
    &\leq C(N,\epsilon)R_{0}^{\alpha_{N}(1-\frac{\beta}{N})(1+\epsilon)^{\frac{1}{N}}\left(\frac{1-\varepsilon_{0}}{N\omega_{N}}\right)^{\frac{1}{N-1}}}\int_{0}^{R_{0}}r^{N-1-\alpha_{N}(1-\frac{\beta}{N})(1+\epsilon)^{\frac{1}{N}}\left(\frac{1-\varepsilon_{0}}{N\omega_{N}}\right)^{\frac{1}{N-1}}-\beta}\;\mathrm{d}r.\notag
\end{align}
We always can choose $\epsilon>0$ such that
$$
N-\alpha_{N}(1-\frac{\beta}{N})(1+\epsilon)^{\frac{1}{N}}\left(\frac{1-\varepsilon_{0}}{N\omega_{N}}\right)^{\frac{1}{N-1}}-\beta>0.
$$
So,
\begin{align}\label{5.8}
    I&\leq C(N)R_{0}^{\alpha_{N}(1-\frac{\beta}{N})(1+\epsilon)^{\frac{1}{N}}\left(\frac{1-\varepsilon_{0}}{N\omega_{N}}\right)^{\frac{1}{N-1}}}\int_{0}^{R_{0}}r^{N-1-\alpha_{N}(1-\frac{\beta}{N})(1+\epsilon)^{\frac{1}{N}}\left(\frac{1-\varepsilon_{0}}{N\omega_{N}}\right)^{\frac{1}{N-1}}-\beta}\;\mathrm{d}r\notag\\
    &\leq C(N)(N-\alpha_{N}(1-\frac{\beta}{N})(1+\epsilon)^{\frac{1}{N}}\left(\frac{1-\varepsilon_{0}}{N\omega_{N}}\right)^{\frac{1}{N-1}}-\beta)^{-1}R_{0}^{N-\beta}\notag\\
    &=C(N,\beta)\left(\int_{B_{R_{0}}}1\;\mathrm{d}x\right)^{1-\frac{\beta}{N}}\notag\\
    &\leq C(N,\beta)\|u\|_{q}^{q(1-\frac{\beta}{N})}.
\end{align}

{\bfseries Case 2:} $0<R_{1}<R_{0}$. Writen as
$$
I=I_{1}+I_{2},
$$
where
\begin{align}
    &I_{1}=\int_{B_{R_{1}}}\frac{\Phi_{N,q,\beta}(\alpha_{N}(1-\frac{\beta}{N})\lvert u\rvert^{\frac{N}{N-1}})}{(1+d\gamma^{-\frac{p}{N-1}(1-\frac{\beta}{N})}\lvert u\rvert^{\frac{p}{N-1}(1-\frac{\beta}{N})})\lvert x\rvert^{\beta}}\;\mathrm{d}x,\notag\\
    &I_{2}=\int_{B_{R_{0}}\setminus B_{R_{1}}}\frac{\Phi_{N,q,\beta}(\alpha_{N}(1-\frac{\beta}{N})\lvert u\rvert^{\frac{N}{N-1}})}{(1+d\gamma^{-\frac{p}{N-1}(1-\frac{\beta}{N})}\lvert u\rvert^{\frac{p}{N-1}(1-\frac{\beta}{N})})\lvert x\rvert^{\beta}}\;\mathrm{d}x.\notag
\end{align}
By using \eqref{5.5}, for $R_{1}\leq r\leq R_{0}$,
$$
u(r)\leq 1+\left(\frac{\varepsilon_{0}}{N\omega_{N}}\right)^{\frac{1}{N}}\left(\ln\frac{R_{0}}{r}\right)^{\frac{N-1}{N}}.
$$
Hence,
$$
u^{\frac{N}{N-1}}(r)\leq(1+\epsilon)^{\frac{1}{N}}\left(\frac{\varepsilon_{0}}{N\omega_{N}}\right)^{\frac{1}{N-1}}\ln\frac{R_{0}}{r}+C(N,\epsilon).
$$
We always can choose $\epsilon>0$ such that
$$
N-\alpha_{N}(1-\frac{\beta}{N})(1+\epsilon)^{\frac{1}{N}}\left(\frac{\varepsilon_{0}}{N\omega_{N}}\right)^{\frac{1}{N-1}}-\beta>0.
$$
So,
\begin{align}
    I_{2}\leq&\int_{B_{R_{0}}\setminus B_{R_{1}}}\frac{\exp(\alpha_{N}(1-\frac{\beta}{N})\lvert u\rvert^{\frac{N}{N-1}})}{\lvert x\rvert^{\beta}}\;\mathrm{d}x\notag\\
    \leq&\int_{B_{R_{0}}\setminus B_{R_{1}}}\frac{\exp(\alpha_{N}(1-\frac{\beta}{N})(1+\epsilon)^{\frac{1}{N}}\left(\frac{\varepsilon_{0}}{N\omega_{N}}\right)^{\frac{1}{N-1}}\ln\frac{R_{0}}{r}+\alpha_{N}(1-\frac{\beta}{N})C(N,\epsilon))}{\lvert x\rvert^{\beta}}\;\mathrm{d}x\notag\\
    \leq& C(N,\epsilon)R_{0}^{\alpha_{N}(1-\frac{\beta}{N})(1+\epsilon)^{\frac{1}{N}}\left(\frac{\varepsilon_{0}}{N\omega_{N}}\right)^{\frac{1}{N-1}}}\int_{R_{1}}^{R_{0}}r^{N-1-\alpha_{N}(1-\frac{\beta}{N})(1+\epsilon)^{\frac{1}{N}}\left(\frac{\varepsilon_{0}}{N\omega_{N}}\right)^{\frac{1}{N-1}}-\beta}\;\mathrm{d}r\notag\\
    \leq& C(N)R_{0}^{\alpha_{N}(1-\frac{\beta}{N})(1+\epsilon)^{\frac{1}{N}}\left(\frac{\varepsilon_{0}}{N\omega_{N}}\right)^{\frac{1}{N-1}}}\notag\\
    &\cdot\frac{R_{0}^{N-\alpha_{N}(1-\frac{\beta}{N})(1+\epsilon)^{\frac{1}{N}}\left(\frac{\varepsilon_{0}}{N\omega_{N}}\right)^{\frac{1}{N-1}}-\beta}-R_{1}^{N-\alpha_{N}(1-\frac{\beta}{N})(1+\epsilon)^{\frac{1}{N}}\left(\frac{\varepsilon_{0}}{N\omega_{N}}\right)^{\frac{1}{N-1}}-\beta}}{N-\alpha_{N}(1-\frac{\beta}{N})(1+\epsilon)^{\frac{1}{N}}\left(\frac{\varepsilon_{0}}{N\omega_{N}}\right)^{\frac{1}{N-1}}-\beta}\notag\\
    \leq& \frac{C(N)}{N-\alpha_{N}(1-\frac{\beta}{N})(1+\epsilon)^{\frac{1}{N}}\left(\frac{\varepsilon_{0}}{N\omega_{N}}\right)^{\frac{1}{N-1}}-\beta}(R_{0}^{N-\beta}-R_{1}^{N-\beta}).\notag
\end{align}
Since $(\sum\limits_{i}a_{i})^{t}\leq\sum\limits_{i}a_{i}^{t}$ for any $0\leq t\leq 1$ and $a_{i}>0$,
\begin{align}\label{5.9}
    I_{2}&\leq C(N,\beta)(R_{0}^{N}-R_{1}^{N})^{1-\frac{\beta}{N}}\notag\\
    &=C(N,\beta)\left(\int_{B_{R_{0}}\setminus B_{R_{1}}}1\;\mathrm{d}x\right)^{1-\frac{\beta}{N}}\notag\\
    &\leq C(N,\beta)\|u\|_{q}^{q(1-\frac{\beta}{N})}.
\end{align}
Now we only need to estimate
$$
I_{1}=\int_{B_{R_{1}}}\frac{\Phi_{N,q,\beta}(\alpha_{N}(1-\frac{\beta}{N})\lvert u\rvert^{\frac{N}{N-1}})}{(1+d\gamma^{-\frac{p}{N-1}(1-\frac{\beta}{N})}\lvert u\rvert^{\frac{p}{N-1}(1-\frac{\beta}{N})})\lvert x\rvert^{\beta}}\;\mathrm{d}x
$$
with $u(R_{1})>1$. Define
$$
v(r)=u(r)-u(R_{1})\;\;\;\mathrm{on}\;0\leq r\leq R_{1}.
$$
It is clear that $v\in W_{0}^{1,N}(B_{R_{1}})$ and $\int_{B_{R_{1}}}\lvert\nabla v\rvert^{N}\;\mathrm{d}x\leq 1-\varepsilon_{0}$. Moreover, for $0\leq r\leq R_{1}$,
$$
u^{\frac{N}{N-1}}(r)\leq(1+\epsilon)^{\frac{1}{N}}v^{\frac{N}{N-1}}(r)+C(N,\epsilon)u^{\frac{N}{N-1}}(R_{1}).
$$
So,
\begin{align}
    I_{1}\leq&\int_{B_{R_{1}}}\frac{\exp(\alpha_{N}(1-\frac{\beta}{N})\lvert u\rvert^{\frac{N}{N-1}})}{d\gamma^{-\frac{p}{N-1}(1-\frac{\beta}{N})}\lvert u\rvert^{\frac{p}{N-1}(1-\frac{\beta}{N})}\lvert x\rvert^{\beta}}\;\mathrm{d}x\notag\\
    \leq&\frac{1}{d\gamma^{-\frac{p}{N-1}(1-\frac{\beta}{N})}}\frac{\exp(\alpha_{N}(1-\frac{\beta}{N})C(N,\epsilon)u^{\frac{N}{N-1}}(R_{1}))}{\lvert u(R_{1})\rvert^{\frac{p}{N-1}(1-\frac{\beta}{N})
    }}\notag\cdot\int_{B_{R_{1}}}\frac{\exp(\alpha_{N}(1-\frac{\beta}{N})(1+\epsilon)^{\frac{1}{N}}v^{\frac{N}{N-1}}(r))}{\lvert x\rvert^{\beta}}\;\mathrm{d}x\notag\\
    \leq&\frac{1}{d\gamma^{-\frac{p}{N-1}(1-\frac{\beta}{N})}}\frac{\exp(\alpha_{N}(1-\frac{\beta}{N})C(N,\epsilon)u^{\frac{N}{N-1}}(R_{1}))}{\lvert u(R_{1})\rvert^{\frac{p}{N-1}(1-\frac{\beta}{N})
    }}\int_{B_{R_{1}}}\frac{\exp(\alpha_{N}(1-\frac{\beta}{N})w^{\frac{N}{N-1}}(r))}{\lvert x\rvert^{\beta}}\;\mathrm{d}x,\notag
\end{align}
where $w=(1+\epsilon)^{\frac{N-1}{N^{2}}}v$. Here $w\in W_{0}^{1,N}(B_{R_{1}})$ and
$$
\int_{B_{R_{1}}}\lvert\nabla w\rvert^{N}\;\mathrm{d}x=(1+\epsilon)^{\frac{N-1}{N}}\int_{B_{R_{1}}}\lvert\nabla v\rvert^{N}\;\mathrm{d}x\leq(1+\epsilon)^{\frac{N-1}{N}}(1-\varepsilon_{0})\leq 1,
$$
if we choose $0<\epsilon\leq(\frac{1}{1-\varepsilon_{0}})^{\frac{N}{N-1}}-1$. Hence, by the singular Trudinger-Moser inequality in \cite{2},
$$
\int_{B_{R_{1}}}\frac{\exp(\alpha_{N}(1-\frac{\beta}{N})w^{\frac{N}{N-1}}(r))}{\lvert x\rvert^{\beta}}\;\mathrm{d}x\leq C(N,\beta)\lvert B_{R_{1}}\rvert^{1-\frac{\beta}{N}}\leq C(N,\beta)R_{1}^{N-\beta}.
$$
Notice that $C(N,\epsilon)=(1-(1+\epsilon)^{-\frac{N-1}{N}})^{-\frac{1}{N-1}}$, and
\begin{align}
    &\varepsilon_{0}\leq1-(1+\epsilon)^{-\frac{N-1}{N}},\notag\\
    &h_{0}=N\left(\frac{\omega_{N}}{1-(1+\epsilon)^{-\frac{N-1}{N}}}\right)^{\frac{1}{N}}>1.\notag
\end{align}
By Lemma \ref{3.2},
\begin{align}
    \frac{\exp(\alpha_{N}(1-\frac{\beta}{N})C(N,\epsilon)u^{\frac{N}{N-1}}(R_{1}))}{\lvert u(R_{1})\rvert^{\frac{p}{N-1}(1-\frac{\beta}{N})
    }}R_{1}^{N-\beta}&=\left[\frac{\exp(\alpha_{N}C(N,\epsilon)u^{\frac{N}{N-1}}(R_{1}))}{\lvert u(R_{1})\rvert^{\frac{p}{N-1}
    }}R_{1}^{N}\right]^{1-\frac{\beta}{N}}\notag\\
    &\leq\left[\frac{\exp(\alpha_{N}C(N,\epsilon)u^{\frac{N}{N-1}}(R_{1}))}{\lvert u(R_{1})\rvert^{\frac{q}{N-1}
    }}R_{1}^{N}\right]^{1-\frac{\beta}{N}}.\notag\\
    &\leq\left[C(N,q)C(N,\epsilon)^{q}\int_{R_{1}}^{\infty}\lvert u(r)\rvert^{q}r^{N-1}\;\mathrm{d}r\right]^{1-\frac{\beta}{N}}\notag\\
    &\leq C(N,q,\beta)\|u\|_{q}^{q(1-\frac{\beta}{N})}.\notag
\end{align}
Therefore,
\begin{align}\label{5.10}
    I_{1}\leq C(N,p,q,d,\beta)\|u\|_{q}^{q(1-\frac{\beta}{N})}.
\end{align}
We draw our conclusion from \eqref{5.1} and \eqref{5.6}-\eqref{5.10}.

Now, by using $\{u_{n}\}_{n=1}^{\infty}$ in the proof of Theorem \ref{1.1}, we prove that the inequality fails if $p<q$ or the constant $\lambda_{N}$ is replaced by any $\lambda>\lambda_{N}$.

First, we establish the necessity of $p\geq q$. There is
\begin{align}
    &\frac{1}{\|u_{n}\|_{q}^{q(1-\frac{\beta}{N})}}\int_{\R^{N}}\frac{\Phi_{N,q,\beta}(\lambda_{N}(1-\frac{\beta}{N})u_{n}^{\frac{N}{N-1}})}{(1+d\lvert u_{n}\rvert^{\frac{p}{N-1}(1-\frac{\beta}{N})})F^{0}(x)^{\beta}}\;\mathrm{d}x\notag\\
    \geq&\frac{1}{\|u_{n}\|_{q}^{q(1-\frac{\beta}{N})}}\frac{\Phi_{N,q,\beta}(\lambda_{N}(1-\frac{\beta}{N})(\frac{1}{N\kappa_{N}})^{\frac{1}{N-1}}(\frac{n}{N-\beta}))}{1+d(\frac{1}{N\kappa_{N}})^{\frac{p}{N(N-1)}(1-\frac{\beta}{N})}(\frac{n}{N-\beta})^{\frac{p}{N}(1-\frac{\beta}{N})}}\int_{0}^{e^{-\frac{n}{N-\beta}}}N\kappa_{N}r^{N-1-\beta}\;\mathrm{d}r\notag\\
    =&\frac{1}{\|u_{n}\|_{q}^{q(1-\frac{\beta}{N})}}\frac{N\kappa_{N}\Phi_{N,q,\beta}(n)}{1+d(\frac{1}{N\kappa_{N}})^{\frac{p}{N(N-1)}(1-\frac{\beta}{N})}(\frac{n}{N-\beta})^{\frac{p}{N}(1-\frac{\beta}{N})}}\frac{e^{-n}}{N-\beta}\notag\\
    \sim&n^{\frac{q-p}{N}(1-\frac{\beta}{N})}.\notag
\end{align}
Hence, if $p<q$,
$$
\frac{1}{\|u_{n}\|_{q}^{q(1-\frac{\beta}{N})}}\int_{\R^{N}}\frac{\Phi_{N,q,\beta}(\lambda_{N}(1-\frac{\beta}{N})u_{n}^{\frac{N}{N-1}})}{(1+d\lvert u_{n}\rvert^{\frac{p}{N-1}(1-\frac{\beta}{N})})F^{0}(x)^{\beta}}\;\mathrm{d}x\rightarrow\infty.
$$
Next, for $\lambda>\lambda_{N}$,
\begin{align}
    &\int_{\R^{N}}\frac{\Phi_{N,q,\beta}(\lambda(1-\frac{\beta}{N})u_{n}^{\frac{N}{N-1}})}{(1+d\lvert u_{n}\rvert^{\frac{p}{N-1}(1-\frac{\beta}{N})})F^{0}(x)^{\beta}}\;\mathrm{d}x\notag\\
    \geq&\frac{\Phi_{N,q,\beta}(\lambda(1-\frac{\beta}{N})(\frac{1}{N\kappa_{N}})^{\frac{1}{N-1}}(\frac{n}{N-\beta}))}{1+d(\frac{1}{N\kappa_{N}})^{\frac{p}{N(N-1)}(1-\frac{\beta}{N})}(\frac{n}{N-\beta})^{\frac{p}{N}(1-\frac{\beta}{N})}}\int_{0}^{e^{-\frac{n}{N-\beta}}}N\kappa_{N}r^{N-1-\beta}\;\mathrm{d}r\notag\\
    =&\frac{N\kappa_{N}\Phi_{N,q,\beta}(\frac{\lambda}{\lambda_{N}}n)}{1+d(\frac{1}{N\kappa_{N}})^{\frac{p}{N(N-1)}(1-\frac{\beta}{N})}(\frac{n}{N-\beta})^{\frac{p}{N}(1-\frac{\beta}{N})}}\frac{e^{-n}}{N-\beta}\notag\\
    \sim&\frac{e^{(\frac{\lambda}{\lambda_{N}}-1)n}}{n^{\frac{p}{N}(1-\frac{\beta}{N})}}.\notag
\end{align}
So,
\begin{align}
    \frac{1}{\|u_{n}\|_{q}^{q(1-\frac{\beta}{N})}}\int_{\R^{N}}\frac{\Phi_{N,q,\beta}(\lambda(1-\frac{\beta}{N})u_{n}^{\frac{N}{N-1}})}{(1+d\lvert u_{n}\rvert^{\frac{p}{N-1}(1-\frac{\beta}{N})})F^{0}(x)^{\beta}}\;\mathrm{d}x\rightarrow\infty.\notag
\end{align}
This finishes the proof of Theorem \ref{1.4}.
\end{proof}

\section{Proofs of Theorem \ref{1.5} and \ref{1.6}}

We will first show that Theorem \ref{1.5} is true if and only if Theorem \ref{1.6} is true, and then we will demonstrate Theorem \ref{1.6}.

\begin{proof}
Suppose first that Theorem \ref{1.6} is true. Let $u\in D^{N,q}(\R^{N})$, $\|F(\nabla u)\|_{N}^{a}+\|u\|_{q}^{kN}\leq 1$. Set
$$
v(x)=u(\lambda x),\;\;\;\mathrm{where}\;\lambda=\frac{1}{\|u\|_{q}^{(k-1)\frac{q}{N}}}.
$$
Then
$$
\|F(\nabla v)\|_{N}=\|F(\nabla u)\|_{N},\;\|v\|_{q}=\frac{1}{\lambda^{\frac{N}{q}}}\|u\|_{q}
$$
and
$$
\|F(\nabla v)\|_{N}^{a}+\|v\|_{q}^{N}=\|F(\nabla u)\|_{N}^{a}+\frac{1}{\lambda^{\frac{N^{2}}{q}
}}\|u\|_{q}^{N}\leq 1.
$$
By Theorem \ref{1.6},
\begin{align}
    \int_{\R^{N}}\frac{\Phi_{N,q,\beta}(\lambda_{N}(1-\frac{\beta}{N})u^{\frac{N}{N-1}})}{(1+d\lvert u\rvert^{\frac{p}{N-1}(1-\frac{1}{k})(1-\frac{\beta}{N})})F^{0}(x)^{\beta}}\;\mathrm{d}x&=\lambda^{N-\beta}\int_{\R^{N}}\frac{\Phi_{N,q,\beta}(\lambda_{N}(1-\frac{\beta}{N})v^{\frac{N}{N-1}})}{(1+d\lvert v\rvert^{\frac{p}{N-1}(1-\frac{1}{k})(1-\frac{\beta}{N})})F^{0}(x)^{\beta}}\;\mathrm{d}x\notag\\
    &\leq C\lambda^{N-\beta}\|v\|_{q}^{q(1-\frac{1}{k})(1-\frac{\beta}{N})}\notag\\
    &=C\lambda^{\frac{N-\beta}{k}}\|u\|_{q}^{q(1-\frac{1}{k})(1-\frac{\beta}{N})}\notag\\
    &=C.\notag
\end{align}

Suppose now that Theorem \ref{1.5} is true. Let $v\in D^{N,q}(\R^{N})$, $\|F(\nabla v)\|_{N}^{a}+\|v\|_{q}^{N}\leq 1$. Set
$$
u(x)=v(\lambda x),\;\;\;\mathrm{where}\;\lambda=\|v\|_{q}^{(1-\frac{1}{k})\frac{q}{N}}.
$$
Then
$$
\|F(\nabla u)\|_{N}=\|F(\nabla v)\|_{N},\;\|u\|_{q}=\frac{1}{\lambda^{\frac{N}{q}}}\|v\|_{q}
$$
and
$$
\|F(\nabla u)\|_{N}^{a}+\|u\|_{q}^{kN}=\|F(\nabla v)\|_{N}^{a}+\frac{1}{\lambda^{\frac{kN^{2}}{q}
}}\|v\|_{q}^{kN}\leq 1.
$$
By Theorem \ref{1.5},
\begin{align}
    \int_{\R^{N}}\frac{\Phi_{N,q,\beta}(\lambda_{N}(1-\frac{\beta}{N})v^{\frac{N}{N-1}})}{(1+d\lvert v\rvert^{\frac{p}{N-1}(1-\frac{1}{k})(1-\frac{\beta}{N})})F^{0}(x)^{\beta}}\;\mathrm{d}x&=\lambda^{N-\beta}\int_{\R^{N}}\frac{\Phi_{N,q,\beta}(\lambda_{N}(1-\frac{\beta}{N})u^{\frac{N}{N-1}})}{(1+d\lvert u\rvert^{\frac{p}{N-1}(1-\frac{1}{k})(1-\frac{\beta}{N})})F^{0}(x)^{\beta}}\;\mathrm{d}x\notag\\
    &\leq C\lambda^{N-\beta}\notag\\
    &=C\|v\|_{q}^{q(1-\frac{1}{k})(1-\frac{\beta}{N})}.
\end{align}

Now, we provide a proof for Theorem \ref{1.6}. Let $u\in D^{N,q}(\R^{N})$, $\|F(\nabla u)\|_{N}^{a}+\|u\|_{q}^{N}\leq 1$. By H\"{o}lder's inequality, Theorem \ref{1.3} and Theorem \ref{1.4},
\begin{align}
    &\int_{\{u>1\}}\frac{\Phi_{N,q,\beta}(\lambda_{N}(1-\frac{\beta}{N})u^{\frac{N}{N-1}})}{(1+d\lvert u\rvert^{\frac{p}{N-1}(1-\frac{1}{k})(1-\frac{\beta}{N})})F^{0}(x)^{\beta}}\;\mathrm{d}x\notag\\
    \leq&\int_{\{u>1\}}\left(\frac{\Phi_{N,q,\beta}(\lambda_{N}(1-\frac{\beta}{N})u^{\frac{N}{N-1}})}{F^{0}(x)^{\beta}}\right)^{\frac{1}{k}}\left(\frac{\Phi_{N,q,\beta}(\lambda_{N}(1-\frac{\beta}{N})u^{\frac{N}{N-1}})}{d^{\frac{k}{k-1}}\lvert u\rvert^{\frac{p}{N-1}(1-\frac{\beta}{N})}F^{0}(x)^{\beta}}\right)^{1-\frac{1}{k}}\;\mathrm{d}x\notag\\
    \leq&\left[\int_{\{u>1\}}\frac{\Phi_{N,q,\beta}(\lambda_{N}(1-\frac{\beta}{N})u^{\frac{N}{N-1}})}{F^{0}(x)^{\beta}}\;\mathrm{d}x\right]^{\frac{1}{k}}\frac{1}{d}\left[\int_{\{u>1\}}\frac{\Phi_{N,q,\beta}(\lambda_{N}(1-\frac{\beta}{N})u^{\frac{N}{N-1}})}{\lvert u\rvert^{\frac{p}{N-1}(1-\frac{\beta}{N})}F^{0}(x)^{\beta}}\;\mathrm{d}x\right]^{1-\frac{1}{k}}\notag\\
    \leq& C(N,p,q,\beta,a,d,k)\|u\|_{q}^{q(1-\frac{\beta}{N})}\notag\\
    \leq&C(N,p,q,\beta,a,d,k)\|u\|_{q}^{q(1-\frac{1}{k})(1-\frac{\beta}{N})}.\notag
\end{align}
Similar to \eqref{5.1} and \eqref{5.6},
\begin{align}
    \int_{\{u\leq 1\}}\frac{\Phi_{N,q,\beta}(\lambda_{N}(1-\frac{\beta}{N})u^{\frac{N}{N-1}})}{(1+d\lvert u\rvert^{\frac{p}{N-1}(1-\frac{1}{k})(1-\frac{\beta}{N})})F^{0}(x)^{\beta}}\;\mathrm{d}x&\leq C(N,p,q,\beta,a,d,k)\|u\|_{q}^{q(1-\frac{\beta}{N})}\notag\\
    &\leq C(N,p,q,\beta,a,d,k)\|u\|_{q}^{q(1-\frac{1}{k})(1-\frac{\beta}{N})}.\notag
\end{align}
So,
$$
\int_{\R^{N}}\frac{\Phi_{N,q,\beta}(\lambda_{N}(1-\frac{\beta}{N})u^{\frac{N}{N-1}})}{(1+d\lvert u\rvert^{\frac{p}{N-1}(1-\frac{1}{k})(1-\frac{\beta}{N})})F^{0}(x)^{\beta}}\;\mathrm{d}x\leq C(N,p,q,\beta,a,d,k)\|u\|_{q}^{q(1-\frac{1}{k})(1-\frac{\beta}{N})}.
$$

Next, by using $\{u_{n}\}_{n=1}^{\infty}$ in the proof of Theorem \ref{1.1}, we prove that the inequality fails if $p<q$ or the constant $\lambda_{N}$ is replaced by any $\lambda>\lambda_{N}$.

Suppose that $p\geq 1$ is such that
\begin{align}
B=&\sup\limits_{\|F(\nabla u)\|_{N}^{a}+\|u\|_{q}^{N}\leq 1}\frac{1}{\|u\|_{q}^{q(1-\frac{1}{k})(1-\frac{\beta}{N})}}\notag\\
&\int_{\R^{N}}\frac{\Phi_{N,q,\beta}(\lambda_{N}(1-\frac{\beta}{N})u^{\frac{N}{N-1}})}{(1+d[\lambda_{N}(1-\frac{\beta}{N})]^{\frac{p}{N}(1-\frac{1}{k})(1-\frac{\beta}{N})}\lvert u\rvert^{\frac{p}{N-1}(1-\frac{1}{k})(1-\frac{\beta}{N})})F^{0}(x)^{\beta}}\;\mathrm{d}x\notag
\end{align}
is finite. Define
\begin{align}
    A_{\lambda}&=\sup\limits_{\|F(\nabla u)\|_{N}\leq 1}\frac{1}{\|u\|_{q}^{q(1-\frac{\beta}{N})}}\int_{\R^{N}}\frac{\Phi_{N,q,\beta}(\lambda(1-\frac{\beta}{N})\lvert u\rvert^{\frac{N}{N-1}})}{(1+d[\lambda(1-\frac{\beta}{N})]^{\frac{p}{N}(1-\frac{1}{k})(1-\frac{\beta}{N})}\lvert u\rvert^{\frac{p}{N-1}(1-\frac{1}{k})(1-\frac{\beta}{N})})F^{0}(x)^{\beta}}\;\mathrm{d}x\notag\\
    &=\sup\limits_{\|F(\nabla u)\|_{N}\leq 1,\;\|u\|_{q}=1}\int_{\R^{N}}\frac{\Phi_{N,q,\beta}(\lambda(1-\frac{\beta}{N})\lvert u\rvert^{\frac{N}{N-1}})}{(1+d[\lambda(1-\frac{\beta}{N})]^{\frac{p}{N}(1-\frac{1}{k})(1-\frac{\beta}{N})}\lvert u\rvert^{\frac{p}{N-1}(1-\frac{1}{k})(1-\frac{\beta}{N})})F^{0}(x)^{\beta}}\;\mathrm{d}x.\notag
\end{align}
Let $u\in D^{N,q}(\R^{N})$, $\|F(\nabla u)\|_{N}\leq 1$, $\|u\|_{q}=1$. Set
$$
v(x)=\left(\frac{\lambda}{\lambda_{N}}\right)^{\frac{N-1}{N}}u(tx),\;\;\;\mathrm{where}\;t=\left(\frac{(\frac{\lambda}{\lambda_{N}})^{N-1}}{1-(\frac{\lambda}{\lambda_{N}})^{a\frac{N-1}{N}}}\right)^{\frac{q}{N^{2}}}.
$$
Then
$$
\|F(\nabla v)\|_{N}^{a}+\|v\|_{q}^{N}\leq 1.
$$
By the definition of $B$,
\begin{align}
    &\int_{\R^{N}}\frac{\Phi_{N,q,\beta}(\lambda(1-\frac{\beta}{N})\lvert u\rvert^{\frac{N}{N-1}})}{(1+d[\lambda(1-\frac{\beta}{N})]^{\frac{p}{N}(1-\frac{1}{k})(1-\frac{\beta}{N})}\lvert u\rvert^{\frac{p}{N-1}(1-\frac{1}{k})(1-\frac{\beta}{N})})F^{0}(x)^{\beta}}\;\mathrm{d}x\notag\\
    =&\int_{\R^{N}}\frac{\Phi_{N,q,\beta}(\lambda(1-\frac{\beta}{N})\lvert u(tx)\rvert^{\frac{N}{N-1}})}{(1+d[\lambda(1-\frac{\beta}{N})]^{\frac{p}{N}(1-\frac{1}{k})(1-\frac{\beta}{N})}\lvert u(tx)\rvert^{\frac{p}{N-1}(1-\frac{1}{k})(1-\frac{\beta}{N})})F^{0}(tx)^{\beta}}\;\mathrm{d}(tx)\notag\\
    =&t^{N-\beta}\int_{\R^{N}}\frac{\Phi_{N,q,\beta}(\lambda_{N}(1-\frac{\beta}{N})\lvert v\rvert^{\frac{N}{N-1}})}{(1+d[\lambda_{N}(1-\frac{\beta}{N})]^{\frac{p}{N}(1-\frac{1}{k})(1-\frac{\beta}{N})}\lvert v\rvert^{\frac{p}{N-1}(1-\frac{1}{k})(1-\frac{\beta}{N})})F^{0}(x)^{\beta}}\;\mathrm{d}x\notag\\
    \leq&\left(\frac{\lambda}{\lambda_{N}}\right)^{\frac{N-1}{N}q(1-\frac{1}{k})(1-\frac{\beta}{N})}\left(\frac{(\frac{\lambda}{\lambda_{N}})^{N-1}}{1-(\frac{\lambda}{\lambda_{N}})^{a\frac{N-1}{N}}}\right)^{\frac{q}{N}\frac{1}{k}(1-\frac{\beta}{N})}B.\notag
\end{align}
As a consequence,
\begin{align}\label{6.2}
   \limsup\limits_{\lambda\uparrow\lambda_{N}}A_{\lambda}\left(1-\frac{\lambda}{\lambda_{N}}\right)^{\frac{q}{N}\frac{1}{k}(1-\frac{\beta}{N})}<\infty.
\end{align}
Meanwhile,
\begin{align}
    &\int_{\R^{N}}\frac{\Phi_{N,q,\beta}(\lambda(1-\frac{\beta}{N})\lvert u_{n}\rvert^{\frac{N}{N-1}})}{(1+d[\lambda(1-\frac{\beta}{N})]^{\frac{p}{N}(1-\frac{1}{k})(1-\frac{\beta}{N})}\lvert u_{n}\rvert^{\frac{p}{N-1}(1-\frac{1}{k})(1-\frac{\beta}{N})})F^{0}(x)^{\beta}}\;\mathrm{d}x\notag\\
    \geq&\frac{N\kappa_{N}\Phi_{N,q,\beta}(\lambda(1-\frac{\beta}{N})(\frac{1}{N\kappa_{N}})^{\frac{1}{N-1}}(\frac{n}{N-\beta}))}{(1+d[\lambda(1-\frac{\beta}{N})]^{\frac{p}{N}(1-\frac{1}{k})(1-\frac{\beta}{N})}\lvert (\frac{1}{N\kappa_{N}})^{\frac{1}{N}}(\frac{n}{N-\beta})^{\frac{N-1}{N}}\rvert^{\frac{p}{N-1}(1-\frac{1}{k})(1-\frac{\beta}{N})})}\int_{0}^{e^{-\frac{n}{N-\beta}}}r^{N-1-\beta}\;\mathrm{d}r\notag\\
    =&\frac{N\kappa_{N}\Phi_{N,q,\beta}(\frac{\lambda}{\lambda_{N}}n)}{(1+d[\lambda(1-\frac{\beta}{N})]^{\frac{p}{N}(1-\frac{1}{k})(1-\frac{\beta}{N})}\lvert (\frac{1}{N\kappa_{N}})^{\frac{1}{N}}(\frac{n}{N-\beta})^{\frac{N-1}{N}}\rvert^{\frac{p}{N-1}(1-\frac{1}{k})(1-\frac{\beta}{N})})}\frac{e^{-n}}{N-\beta}\notag\\
    \sim&\frac{e^{(\frac{\lambda}{\lambda_{N}}-1)n}}{n^{\frac{p}{N}(1-\frac{1}{k})(1-\frac{\beta}{N})}}.\notag
\end{align}
Let us pick the subsequence $\{u_{n_{\lambda}}\}$ of $\{u_{n}\}$ which satisfies
$$
\frac{1}{n_{\lambda}}\sim 1-\frac{\lambda}{\lambda_{N}}\;\;\;\mathrm{for}\;\lambda\lesssim\lambda_{N}.
$$
After that,
\begin{align}
    A_{\lambda}&\geq\frac{1}{\|u_{n_{\lambda}}\|_{q}^{q(1-\frac{\beta}{N})}}\int_{\R^{N}}\frac{\Phi_{N,q,\beta}(\lambda(1-\frac{\beta}{N})\lvert u_{n_{\lambda}}\rvert^{\frac{N}{N-1}})}{(1+d[\lambda(1-\frac{\beta}{N})]^{\frac{p}{N}(1-\frac{1}{k})(1-\frac{\beta}{N})}\lvert u_{n_{\lambda}}\rvert^{\frac{p}{N-1}(1-\frac{1}{k})(1-\frac{\beta}{N})})F^{0}(x)^{\beta}}\;\mathrm{d}x\notag\\
    &\gtrsim\frac{n_{\lambda}^{\frac{q}{N}(1-\frac{\beta}{N})}e^{(\frac{\lambda}{\lambda_{N}}-1)n_{\lambda}}}{n_{\lambda}^{\frac{p}{N}(1-\frac{1}{k})(1-\frac{\beta}{N})}}\notag\\
    &\sim\left(\frac{1}{1-\frac{\lambda}{\lambda_{N}}}\right)^{(\frac{q}{N}-\frac{p}{N}(1-\frac{1}{k}))(1-\frac{\beta}{N})}.\notag
\end{align}
As a consequence,
\begin{align}\label{6.3}
    \liminf\limits_{\lambda\uparrow\lambda_{N}}A_{\lambda}\left(1-\frac{\lambda}{\lambda_{N}}\right)^{(\frac{q}{N}-\frac{p}{N}(1-\frac{1}{k}))(1-\frac{\beta}{N})}>0.
\end{align}
By \eqref{6.2} and \eqref{6.3},
$$
\limsup\limits_{\lambda\uparrow\lambda_{N}}\frac{(1-\frac{\lambda}{\lambda_{N}})^{\frac{q}{N}\frac{1}{k}(1-\frac{\beta}{N})}}{(1-\frac{\lambda}{\lambda_{N}})^{(\frac{q}{N}-\frac{p}{N}(1-\frac{1}{k}))(1-\frac{\beta}{N})}}<\infty,
$$
or equivalently
$$
p\geq q.
$$
Finally, we establish the sharpness of $\lambda_{N}$. Set
$$
v_{n}=c_{n}u_{n},
$$
where $c_{n}>0$ are chosen such that
$$
c_{n}^{a}+c_{n}^{N}\|u_{n}\|_{q}^{N}=1.
$$
Then
$$
\|F(\nabla v_{n})\|_{N}^{a}+\|v_{n}\|_{q}^{N}=1.
$$
Note that
$$
\mathrm{lim}_{n\rightarrow\infty}c_{n}=1.
$$
Hence, for $\lambda>\lambda_{N}$,
\begin{align}
    &\int_{\R^{N}}\frac{\Phi_{N,q,\beta}(\lambda(1-\frac{\beta}{N})v_{n}^{\frac{N}{N-1}})}{(1+d\lvert v_{n}\rvert^{\frac{p}{N-1}(1-\frac{1}{k})(1-\frac{\beta}{N})})F^{0}(x)^{\beta}}\;\mathrm{d}x\notag\\
    =&\int_{\R^{N}}\frac{\Phi_{N,q,\beta}(\lambda(1-\frac{\beta}{N})(c_{n}u_{n})^{\frac{N}{N-1}})}{(1+d\lvert c_{n}u_{n}\rvert^{\frac{p}{N-1}(1-\frac{1}{k})(1-\frac{\beta}{N})})F^{0}(x)^{\beta}}\;\mathrm{d}x\notag\\
    \geq&\frac{\Phi_{N,q,\beta}(\lambda(1-\frac{\beta}{N})c_{n}^{\frac{N}{N-1}}(\frac{1}{N\kappa_{N}})^{\frac{1}{N-1}}(\frac{n}{N-\beta}))}{1+dc_{n}^{\frac{p}{N-1}(1-\frac{1}{k})(1-\frac{\beta}{N})}(\frac{1}{N\kappa_{N}})^{\frac{p}{N(N-1)}(1-\frac{1}{k})(1-\frac{\beta}{N})}(\frac{n}{N-\beta})^{\frac{p}{N}(1-\frac{1}{k})(1-\frac{\beta}{N})}}\int_{0}^{e^{-\frac{n}{N-\beta}}}N\kappa_{N}r^{N-1-\beta}\;\mathrm{d}r\notag\\
    =&\frac{N\kappa_{N}\Phi_{N,q,\beta}(\frac{\lambda}{\lambda_{N}}nc_{n}^{\frac{N}{N-1}})}{1+dc_{n}^{\frac{p}{N-1}(1-\frac{1}{k})(1-\frac{\beta}{N})}(\frac{1}{N\kappa_{N}})^{\frac{p}{N(N-1)}(1-\frac{1}{k})(1-\frac{\beta}{N})}(\frac{n}{N-\beta})^{\frac{p}{N}(1-\frac{1}{k})(1-\frac{\beta}{N})}}\frac{e^{-n}}{N-\beta}\notag\\
    \sim&\frac{e^{(\frac{\lambda}{\lambda_{N}}c_{n}^{\frac{N}{N-1}}-1)n}}{n^{\frac{p}{N}(1-\frac{1}{k})(1-\frac{\beta}{N})}}\notag\\
    \sim&\frac{e^{(\frac{\lambda}{\lambda_{N}}-1)n}}{n^{\frac{p}{N}(1-\frac{1}{k})(1-\frac{\beta}{N})}}.\notag
\end{align}
Therefore,
\begin{align}
    \frac{1}{\|v_{n}\|_{q}^{q(1-\frac{1}{k})(1-\frac{\beta}{N})}}\int_{\R^{N}}\frac{\Phi_{N,q,\beta}(\lambda(1-\frac{\beta}{N})v_{n}^{\frac{N}{N-1}})}{(1+d\lvert v_{n}\rvert^{\frac{p}{N-1}(1-\frac{1}{k})(1-\frac{\beta}{N})})F^{0}(x)^{\beta}}\;\mathrm{d}x\rightarrow\infty.\notag
\end{align}
This finishes the proof of Theorem \ref{1.6}.
\end{proof}

\def\refname{References }

\end{document}